\documentclass[10pt]{amsart}
\usepackage[margin=1.2in]{geometry}
\usepackage{amssymb,mathrsfs}
\usepackage{amsmath}
\usepackage{amsopn}
\usepackage{combelow}

 \usepackage[backend=biber,style=alphabetic]{biblatex}
 \bibliography{Bibtex}
 \usepackage{hyperref}
 \usepackage{xurl}
 \hypersetup{
 colorlinks = true,
 linkcolor = blue,
 breaklinks = true
 }

\newcommand{\diffto}{\xrightarrow{\raisebox{-0.2 em}[0pt][0pt]{\smash{\ensuremath{\sim}}}}}

\newcommand{\dif}{\mathrm{d}}

\newcommand{\SD}{{\mathcal{D}}}
\newcommand{\X}{{\mathfrak{X}}}
\newcommand{\Lie}{{\mathcal{L}}}
\newcommand{\F}{{\mathcal{F}}}
\newcommand{\SI}{{\mathcal{I}}}
\newcommand{\SJ}{{\mathcal{J}}}
\newcommand{\SR}{{\mathcal{R}}}
\newcommand{\SM}{{\mathcal{M}}}

\newcommand{\R}{{\mathbb{R}}}
\newcommand{\Z}{{\mathbb{Z}}}
\newcommand{\C}{{\mathbb{C}}}
\newcommand{\N}{\mathbb{N}}
\newcommand{\wtd}{\widetilde}
\renewcommand{\d}{{\operatorname{d}}}
\DeclareMathOperator{\im}{Im}

\newtheorem*{definition*}{Definition}

\newtheorem{lemma}{Lemma}[section]
\newtheorem{proposition}[lemma]{Proposition}
\newtheorem{theorem}[lemma]{Theorem}
\newtheorem{corollary}[lemma]{Corollary}
\newtheorem{definition}[lemma]{Definition}

\newtheorem{remark}[lemma]{Remark}
\newtheorem{example}[lemma]{Example}
\newtheorem{question}[lemma]{Question}
\newtheorem{claim}[lemma]{Claim}

\newcommand{\Addresses}{{
  \bigskip
  \footnotesize

Lauran Toussaint,\par\vspace*{-10pt}
\textsc{Vrije Universiteit Amsterdam, 1081 HV Amsterdam, The Netherlands}\par\vspace*{-10pt}
  \textit{E-mail address}: \texttt{l.e.toussaint@vu.nl}

\medskip

 Florian Zeiser,\par\vspace*{-10pt}\nopagebreak
\textsc{University of Illinois at Urbana-Champaign, 61801 Urbana, United States}\par\vspace*{-10pt}\nopagebreak
\textit{E-mail address}: \texttt{fzeiser@illinois.edu}
}}

\title{Formal Poisson (co)homology of the Lefschetz singularity}
\author{Lauran Toussaint, Florian Zeiser}

\begin{document} 

\maketitle

\begin{abstract}
We compute the formal Poisson cohomology groups of a real Poisson structure $\pi$ on $\C^2$ associated to the Lefschetz singularity $(z_1,z_2)\mapsto z_1^2 + z_2^2$. In particular we correct an erroneous computation in the literature. 
The definition of $\pi$ depends on a choice of volume form. Using the main result we formally classify all Poisson structure arising from different choices of volume forms.
\end{abstract}

\section{Introduction}
A Poisson structure on a smooth manifold $M$ is a bivector field $\pi \in \mathfrak{X}^2(M) $ such that
\begin{equation*}
    [\pi ,\pi ]=0,
\end{equation*}
where $[\cdot ,\cdot ]$ denotes the Schouten bracket (see e.g. \cite{LPV13}). As first observed by Lichnerowicz \cite{Lic77}, any Poisson manifold $(M,\pi)$ is naturally equipped with a differential on the space of mulivector fields
\[ d_\pi := [\pi,\cdot]:\ \X^\bullet(M) \to \X^{\bullet +1}(M). \]
The vanishing of the bracket is equivalent to $\d_\pi^2 =0$, and the associated cohomology $H^\bullet(M,\pi)$ is called \emph{Poisson cohomology}. The groups $H^\bullet(M,\pi)$ contain important information about the Poisson manifold $(M,\pi)$. For example:
\begin{itemize}
    \item $H^0(M,\pi)$ consists of functions on $M$ constant along the leaves of the foliation $\F$;
    \item $H^1(M,\pi),[\cdot,\cdot])$ can be seen as the Lie algebra of infinitesimal outer automorphisms of $(M,\pi)$;
    \item $H^2(M,\pi)$ controls infinitesimal deformations modulo those induced by diffeomorphisms.
\end{itemize}
For a detailed exposition of Poisson geometry and Poisson cohomology we refer the reader to \cite{DZ05,LPV13,CFM21}. Despite their importance, Poisson cohomology is notoriously hard to compute due to the lack of general methods for the computation. The difficulty comes from the fact that $(\mathfrak{X}^{\bullet}(M),\dif_{\pi})$ is generally not an elliptic complex. Assuming that $\pi$ is regular, i.e. all leaves of $\F$ have the same dimension, Xu computed the Poisson cohomology of such Poisson structures in \cite{Xu92}. This, together with the existence of a Mayer-Vietoris sequence due to Vaisman \cite{Vai94}, implies that the fundamental problem for the computation lies around neighborhoods of singular leaves. 

To this end observe that given a Poisson structure $\pi$ on $M$ and a singular point $p\in M$ of $\pi$, i.e. $\pi_p = 0$, we have a short exact sequence of complexes
\begin{align*}
    0\to (\mathfrak{X}^\bullet_p(M),\dif_{\pi})\hookrightarrow (\mathfrak{X}^\bullet(M),\dif_{\pi}) \xrightarrow{j^\infty_p} (\mathfrak{X}^\bullet(M)/\mathfrak{X}^\bullet_p(M)\simeq \R[[T^*_pM]]\otimes \wedge^{\bullet}T_pM,\dif_{j_p^{\infty}\pi} )\to 0
\end{align*}
resulting in a long exact sequence in cohomology. Here $j^\infty_p$ is the infinite jet map at $p$, and $\mathfrak{X}^\bullet_p(M)$ denotes multivector fields whose infinite jet at $p$ is zero. The equivalence for the quotient is due to Borel's lemma. A similar short exact sequence can be obtained when the singular locus of $\pi$ is a higher dimensional submanifold, see \cite{HZ23}.

The short exact sequence above was first used by Abreu \& Ginzburg \cite{Gin96} and Roytenberg \cite{Roy02}. Its usefulness lies in the fact that if $\pi$ vanishes in a polynomial way then we can compute the cohomology of $(\mathfrak{X}^\bullet_p(M),\dif_{\pi})$ as if $\pi$ was regular. As such we are left with the computation of the complex
\[ (\R[[T^*_pM]]\otimes \wedge^{\bullet}T_pM,\dif_{j_p^{\infty}\pi} ). \]
We call its cohomology the formal Poisson cohomology of $\pi$ at $p\in M$, denoted by $H^\bullet_{F_p}(M,\pi)$. The methods used to compute $H^\bullet_{F_p}(M,\pi)$ are more algebraic in nature (see e.g. \cite{Mon02, Pic06, Pel09}). 

The Poisson structure under consideration in this paper are defined as follows. First recall that any fibration $f:(M^m,\mu_M) \to (N^{m-2},\mu_N)$ between oriented manifolds $(M,\mu_M)$ and $(N,\mu_N)$ defines a Poisson structure $\pi_{f,\mu_M,\mu_N}$ on $M$ due to \cite{GSV14}[Thm 2.6] by:
\begin{align}\label{eq: Poisson via fibration}
    \pi_{f,\mu_M,\mu_N}(\dif g,\dif h)\mu_M := \dif g\wedge \dif h\wedge f^*{\mu_{N}} \qquad \text{ for }\quad g,h\in C^{\infty}(M).
\end{align}
Note that $\pi_{f,\mu_M,\mu_N}$ depends on the fibration $f:M\to N$ and on the choice of volume forms on $M$ and $N$. Poisson structures defined by different volume forms admit the same singular foliation $\F$ with:
\begin{itemize}
    \item $2$-dimensional leaves through every $x\in M\setminus \mathrm{Crit}(f)$;
    \item $0$-dimensional leaves for every $x\in \mathrm{Crit}(f)$.
\end{itemize}
However, the symplectic form along the leaves will vary for different choices of volume forms. The study of such Poisson structures has been of increased interest lately (see e.g. \cite{STV19,BO22}). Our Poisson structure of interest is induced by a fibration with a \emph{Lefschetz singulartiy}:
\begin{align}\label{eq: lf}
    f:\R^4 \to \R^2 \qquad \text{ with }\quad f= (f_1,f_2) =\left(x_1^2-x_2^2+x_3^2-x_4^2 ,\,  2(x_1x_2+x_3x_4)\right)
\end{align}
\begin{definition}\label{def: poisson biv}
Let $\pi=\pi_{f,\mu_M,\mu_N}$ be the Poisson structure on $\R^4$ defined by \eqref{eq: Poisson via fibration}, \eqref{eq: lf} and the standard volume forms in $\R^4$ and $\R^2$, respectively.
\end{definition}
In this paper we take the first step in the computation of $H^\bullet(\R^4,\pi)$ by computing the formal Poisson cohomology of $\pi$ at the origin which we denote by $H^\bullet_F(\R^4,\pi)$. We do so by means of its Poisson homology, i.e.\ on a Poisson manifold $(M^m,\pi)$ we have another differential on the space of forms
\begin{align*}
    \delta_{\pi}:\Omega^{\bullet}(M)\to \Omega^{\bullet-1}(M) \qquad
    \delta_{\pi} := \iota_{\pi}\circ \dif -\dif \circ \iota_{\pi}
\end{align*}
introduced by Kozsul \cite{Kos85} and Brylinski \cite{Bry88}. Here $\dif$ denotes the de-Rham differential and $\iota$ the contraction between mulitvector fields and forms. The homology associated to $(\Omega^{\bullet}(M),\delta_{\pi})$ is called Poisson homology and denoted by $H_{\bullet}(M,\pi)$. Given a volume form $\mu\in \Omega^{m}(M)$ on $M$ the relation between Poisson (co)homology can be encoded as follows. First note that $\mu$ induces the isomorphism 
\begin{align}\label{eq: isomorphism form vf}
    \star := \mu^\flat :\mathfrak{X}^{\bullet}(M)\diffto\Omega^{\bullet}(M), \qquad X\mapsto \iota_X(\mu).
\end{align}
We define a vector field $X_{\mu} \in \mathfrak{X}^1(M)$, the \emph{modular vector field} of $(M,\pi) $ and $\mu$ (see \cite{Wei97}), by
\begin{equation*}
    \star X_\mu := \dif \star (\pi).
\end{equation*}
The vector field $X_{\mu}$ is Poisson, i.e. $\dif_{\pi}X_{\mu}=0$ and the class $\textrm{mod}(M,\pi)=[X_\mu]\in H^1(M,\pi)$ is independent of the chosen volume form $\mu$. 
Using the isomorphism $\star$ we obtain the relation 
\begin{equation}\label{eq:partialpi_identity} 
    \delta_{\pi} =\star \circ (\dif_{\pi}+X_\mu \wedge \cdot ) \circ \star^{-1},
\end{equation}
between the Poisson differentials $\dif_{\pi}$ and $\delta_\pi$ (see \cite{LPV13}[Proposition 4.18]). We call $(M,\pi)$ \emph{unimodular} if $\textrm{mod}(M,\pi)=0\in H^1(M,\pi)$. It is well-known that $(M,\pi)$ is unimodular iff there exists a volume form $\mu$ with $X_\mu=0$. In particular, for all unimodular Poisson structures we obtain an isomorphism
\begin{align}\label{eq: iso cohom}
    H_{\bullet}(M,\pi)\diffto H^{m-\bullet}(M,\pi).
\end{align}
In particular, for $(\R^4, \pi)$ from Definition \ref{def: poisson biv} we have
\[ \star(\pi)=\dif f_1\wedge \dif f_2,\]
and hence $(\R^4, \pi)$ is unimodular. Note that the isomorphism induced by \eqref{eq: isomorphism form vf} descends via $j^\infty_0$ to an isomorphism between formal vector fields and forms, which we denote by
\begin{align*}
    \mathfrak{X}^{\bullet}_f:= \SR\otimes \wedge^{\bullet}\R^4 \qquad \text{ and }\qquad \Omega^{\bullet}_f :=  \SR\otimes \wedge^{\bullet}(\R^4)^*.
\end{align*}
Here $\SR:= \R[[x_1,\dots ,x_4]]$ denotes the ring of formal power series. Therefore, \eqref{eq: iso cohom} descends to the formal setting and we obtain an isomorphism between formal Poisson (co)homology 
\begin{align*}
    H_{\bullet}^F(\R^4,\pi)\diffto H^{4-\bullet}_F(\R^4,\pi),
\end{align*}
which allows us to compute formal Poissson cohomology by means of its homological counterpart. Since from here on out we refer to Poisson (co)homology only in the formal setting, we drop the sub- and superscript $F$, respectively throughout the rest of the paper. 
\newpage
The main result is the following:
\begin{theorem}\label{main theorem}
 The formal Poisson homology groups $H_{\bullet}(\R^4,\pi)$ of $(\R^4,\pi)$ are uniquely described in the various degrees as follows:
\begin{itemize}
    \item in degree $0$, $H_0(\R^4,\pi)$ has unique representatives of the form
    \begin{equation*}
        p +  \sum_{i=1}^4 a_ix_i
    \end{equation*}
    where $p \in  \R[[f_1,f_2]]$ and  $a_i \in \R[[x_2^2,x_4]]$;
    \item the group $H_1(\R^4,\pi)$ has unique representatives of the form
    \begin{equation*}
        \sum_{j=1}^2 p_j \zeta_j + q_j \d f_j + \sum_{i=1}^4 \d (a_ix_i) + b_ix_i \d f_1
    \end{equation*}
    where $p_j,q_j \in  \R[[f_1,f_2]]$ and  $a_i,b_i \in \R[[x_2^2,x_4]]$ and
    \begin{align}\label{eq: dual rot}
        \zeta_1:= \frac{1}{2}(-x_3\dif x_1 +x_4\dif x_2+x_1\dif x_3 - x_2\dif x_4), \qquad \zeta_2:= \frac{1}{2}(-x_4\dif x_1 -x_3\dif x_2+x_2\dif x_3 + x_1\dif x_4);
    \end{align}
    \item in degree $2$, $H_2(\R^4,\pi)$ has unique representatives of the form
    \begin{equation*}
        p \zeta_1\wedge \zeta_2 +q\d f_1\wedge \d f_2 + \sum_{i=1}^2 p_i \d(f_1\zeta_i) + q_i\d \zeta_i +\sum_{i=1}^4 \d (a_ix_i)\wedge \d f_1
    \end{equation*}
    where $p, p_i, q, q_i \in  \R[[f_1,f_2]]$ and $a_i \in \R[[x_2^2,x_4]]$;
    \item representatives of $H_3(\R^4,\pi)$ are uniquely described by
    \begin{align*}
        \sum_{i=1}^2 p_i \zeta_2 \wedge \d \zeta_i +q_i \d f_1 \wedge \d \zeta_i 
    \end{align*}
    where $p_i, q_i \in  \R[[f_1,f_2]]$;
    \item  in degree $4$, elements in $ H_4(\R^4,\pi)$ are of the form
    \begin{align*}
        \R[[f_1,f_2]]\mu.
    \end{align*}
\end{itemize}
\end{theorem}
\begin{remark}\label{rmk:literature_mistake}
The groups $H_\bullet (\R^4,\pi)$ from Theorem \ref{main theorem}, or equivalently $H^\bullet (\R^4,\pi)$, were erroneously computed in \cite{BV20}. We outline the problem of their computation in section \ref{section: mistake explanation}.
\end{remark}

Using the isomorphism \eqref{eq: iso cohom} we describe the corresponding cohomology groups in Section \ref{sec:interpretation results}. In particular, we can use the result for the second formal cohomology to classify Poisson structures arising from the fibration in \eqref{eq: lf} via \eqref{eq: Poisson via fibration}. We obtain the following result.
\begin{corollary}\label{corollary: deformations}
    Any Poisson structure obtained from a different volume form is formally equivalent to a Poisson structure obtained from 
    \[ \mu_p = (c+p)\mu\]
    for $p\in \langle f_1,f_2\rangle_{\R[[f_1,f_2]]} $ and $c\in \R_+$ and any $p$ and $c$.
\end{corollary}
Note that the space of forms $\Omega^\bullet$ comes naturally equipped with another differential, the de-Rham differential $\d$. One has the relation
\begin{align*}
    \dif\circ \delta_{\pi} +\delta_{\pi}\circ \dif =0.
\end{align*}
Therefore we have a bidifferential complex $(\Omega^{\bullet}_f,\dif,\delta_{\pi})$ and an induced differential complex in Poisson homology $(H_{\bullet}(\R^4,\pi),\dif)$. It's cohomology are described as follows.
\begin{corollary}\label{corollary: de-rham cohom}
    For the complex $(H_{\bullet}(\R^4,\pi),\dif)$ we find
    \begin{align*}
        H_{DR}^{k}(H_{\bullet}(\R^4,\pi),\dif)=\begin{cases}
            \R &\text{ if } k=0\\
            0& \text{ else.}
        \end{cases}
    \end{align*}
\end{corollary}

\subsection*{Strategy of the proof of Theorem \ref{main theorem}}

Roughly speaking the proof of Theorem \ref{main theorem} is as follows. We start by computing the kernel of $\delta_{\pi}$. In turn this allows us to compute the dimension of the Poisson homology groups. Together with the Poisson homology groups these kernels fit into short exact sequences. This allows us to compute the Hilbert-Poincare series associated to the formal Poisson homology groups, and hence obtain their dimensions. The last step consists of finding an explicit set of generators of the right dimension.

In principle all the arguments consist of explicit computations. 
However, a recurring complication in these computations is that the following property does not always hold.
Given $\beta\in \Omega^k_f$, does
    \[ \d f_1 \wedge \d f_2 \wedge \beta =0 \quad \text{ imply that } \quad \beta = \sum_{i=1}^2 \d f_i \wedge \alpha_i \quad \text{ for some }\quad \alpha _i \in \Omega^{k-1}_f ?\]
The failure of this property to hold is measured by the groups $\mathcal{D}^k(\d f_1, \d f_2)$, first defined by \cite{Rha54} and generalized by Saito \cite{Sai76}.
It turns out that for $k=1$ this group vanishes. On the other hand $\mathcal{D}^2(\d f_1,\d f_2)$ is non-zero, which plays an important role in the computations. 

There appears to be a close relation between the generators of $\mathcal{D}(\d f_1, \d f_2)$ and those of $H_\bullet(\R^4,\pi)$. For example, $\d \zeta_i$ appear both in Theorem \ref{main theorem}, and as the generators of $\mathcal{D}^2(\d f_1, \d f_2)$, see Proposition \ref{proposition: iso division group}. Similarly, the appearance of the classes represented by coefficients $a_i$ and $b_i$ in $\R[[x_2^2,x_4]]$ appears to be related to the description of the division groups. We do not explore this connection further but leave it as an option question:

\begin{question}
    Let $f_1,\dots , f_{n-2}:\R^n\to \R$ be smooth functions. What is the precise relation between the groups $\mathcal{D}^i(\d f_1,\dots, \d f_{n-2})$ and the formal Poisson (co)homology of the Jacobi-Poisson structure on $\R^n$ induced by the functions $f_1,\dots , f_{n-2}$?
\end{question}

\subsection*{Organization of the paper}

Recall from \eqref{eq: iso cohom} that the formal Poisson homology and cohomology groups are isomorphic. In Section \ref{sec:interpretation results} we start by giving a cohomology description of the main theorem, and a geometric interpretation of some of the generators. We also partially describe the Gerstenhaber algebra structure of the groups $H^\bullet(\R^4,\pi)$.

In Section \ref{sec:preliminaries} we discuss some preliminary notions from Poisson geometry (Section \ref{sec:prel poisson}), and algebra (Section \ref{sec:prel algebra}). Most notably, the definition of Jacobi-Poisson structures and an explicit description of $\delta_\pi$. Most of our computations involve power series, and are done degree wise. To this end we recall, in Section \ref{sec:prel hilbertseries}, the definition of Hilbert-Poincare series, which we later use to compute the rank of the Poisson homology groups in each homogeneous degree. We also recall some standard facts from algebra concerning regular sequences in Section \ref{sec:prel regularsequences}, and standard bases in Section \ref{sec:prel standardbases}.

In Section \ref{sec: division grous} we recall the notion of division groups, which measure the failure for the division property to hold. The main result of this section is Proposition \ref{proposition: iso division group} giving an explicit description of $\mathcal{D}^2(\d f_1,\d f_2)$, the second division group associated to the coefficients of $\d f_1$ and $\d f_2$.

Using these results, Section \ref{sec: ker poisson diff} computes the kernel of the Poisson differential $\delta_\pi$. The results are collected in Proposition \ref{proposition: kernel differential}. The proof of this proposition, which is a long but more or less straightforward computation, is the core of the paper. 
In Section \ref{sec: poincare series} we use the description of $\ker \delta_\pi$ to compute the Hilbert-Poincare series of the Poisson homology groups, see Proposition \ref{proposition: Poincare series}.

The proof of Theorem \ref{main theorem}, which combines the results from Section \ref{sec: ker poisson diff} and Section \ref{sec: poincare series}, is given in Section \ref{sec:proof main thm}. We check that the generators in Theorem \ref{main theorem} are indeed in the kernel. Then, using the Hilbert-Poincare series we see that the set of generators have the right dimension. The remainder of the proof consists of showing that none of the generators are in the image of $\delta_\pi$. 

Lastly, Section \ref{sec:deformations} and Section \ref{sec:de-rham cohom} respectively contain the proofs of Corollary \ref{corollary: deformations} and Corollary \ref{corollary: de-rham cohom}.

\subsection*{Acknowledgements}
We would like to thank Ioan M\u{a}rcu\cb{t} for bringing the problem to our attention and useful discussions. L.\ Toussaint is funded by the Dutch Research Council (NWO) on the project ``proper Fredholm homotopy theory'' (OCENW.M20.195) of the research programm Open Competition ENW M20-3. F.\ Zeiser would like to thank the Max Planck Institute for Mathematics in Bonn for its hospitality and financial support during the early stages of this project and the group of Thomas Rot at Vrije Universiteit Amsterdam for its hospitality and support during a research visit.

\section{An interpretation of the results}\label{sec:interpretation results}

In this section we review theorem \ref{main theorem} in view of \eqref{eq: iso cohom}, stating the corresponding result for formal Poisson cohomology and interpreting the result. In the first section we discuss the foliation $\F$ associated with $\pi$ and its relation with $H^0(\R^4,\pi)$. The second section is devoted to the Lie algebra $H^1(\R^4,\pi)$ and in the third section we look at $H^2(\R^4,\pi)$ and the interpretation of some its classes in terms of deformations of $\pi$. Finally, in the last section we describe the higher Poisson cohomology groups and some of the additional algebraic structure present for Poisson cohomology.  

\subsection{The foliation and the Casimir functions}
For the geometric interpretation let's take a closer look at the foliation $\F$ induced by $\pi$. From the introduction we know that $\F$ is closely related to the fibers of the fibration $f$ from \eqref{eq: lf}. In particular, for $(x,y)\ne 0 \in \R^2$, the fiber $f^{-1}(x,y)$ is connected and diffeomorphic to a cylinder. 
\[ f^{-1}(x,y)\simeq S^1\times \R \]
The preimage of the origin under $f$ consists of three leaves. The set $f^{-1}(0)\setminus \{0\}$ has two connected components both of which are diffeomorphic to cylinders. The origin $0\in \R^4$ is the sole critical point of $f$ and is therefore a leaf of dimension $0$. 
\[ f^{-1}(0) = S^1\times \R \cup \{ 0 \} \cup S^1\times \R\]
The leaf space $\R^4/\F$ is, as a topological space, given by $\R^2$ where we have three distinct points instead of the origin. However, continuous functions on $\R^4/\F$ can not distinguish these three points. Viewing $f_1$ and $f_2$ as the coordinate functions on the leaf space we obtain for formal functions on the leaf space precisely what we expect, i.e./ power series in these coordinates. That is, in degree $0$ we have 
\[ H^0(\R^4,\pi) = \R[[f_1,f_2]].\]
\begin{remark}
    In fact one can use the discussion above to show that we have an isomorphism
    \[ C^{\infty}(\R^2) \to H^0(\R^4, \pi), \quad g\mapsto g(f_1,f_2).\]
    Here $H^0(\R^4, \pi)$ refers to the Poisson cohomology over smooth multivector fields (see \cite{MZ23}).
\end{remark}
\subsection{The Lie algebra $H^1(\R^4,\pi)$}
To describe the Poisson cohomology group in degree one let $E_i$ and $T_i$, $i=1,2$, be the real and imaginary part of the complex vector fields
\begin{align*}
    E=z\partial_z+w\partial_w \qquad \text{ and } \qquad T:= z\partial_w -w\partial_z;
\end{align*}
where $(z,w)=(x_1+ix_2,x_3+ix_4)$.
\begin{corollary}
    In degree one $H^1(\R^4,\pi)$ is the free $H^0(\R^4,\pi)$-module generated by
    \begin{align*}
        E_1, E_2, T_1 ,T_2.
    \end{align*}
\end{corollary}
This follows immediately from the relations
\begin{align}
    \star E_1 = \zeta_i\wedge \d \zeta_i, \quad \star E_2 = -\zeta_1\wedge \d \zeta_2 = \zeta_2 \wedge \d \zeta_1,\qquad \label{eq: rel E, zeta}\\ 
    \star T_1= -\frac{1}{4}\d f_i\wedge \d \zeta_i, \quad \star T_2= -\frac{1}{4}\d f_2\wedge \d \zeta_1 =\frac{1}{4} \d f_1 \wedge \d \zeta_2 \label{eq: rel T, zeta} 
\end{align}
The Lie bracket on $H^1(\R^4,\pi)$ is induced by the Lie bracket for vector fields. In order to describe the Lie algebra structure, we note that 
\[ [E_1,E_2] = [T_1,T_2] =[E_i,T_j]=0.\]
Hence the brackets in cohomology are fully described by the relations
\begin{align}\label{eq: relation E, T and f}
    \Lie_{E_1}(f_i)= f_i  , \quad \Lie_{E_2}f_1=f_2,\quad \Lie_{E_2} f_2=-f_1, \quad \Lie_{T_i}f_j=0.
\end{align}
Geometrically, \eqref{eq: relation E, T and f} means that the Lie algebra
\[ \mathfrak{g}_N:=\{p_iE_i \ | \ p_i \in \R[[f_1,f_2]]\}\]
describes Poisson vector fields who project non-trivially to the leaf space, i.e.\ $E_1$ is mapped to the Euler vector field and $E_2$ to the rotational vector field. The vector fields $T_i$ preserve the leaves and form a commutative Lie subalgebra of Poisson vector fields tangent to $\F$ which we denote by
\[\mathfrak{g}_T:= \{p_iT_i \ | \ p_i \in \R[[f_1,f_2]]\} \]
As a Lie algebra, the first formal Poisson cohomology is given by 
\[H^1(\R^4,\pi)\simeq \mathfrak{g}_T\rtimes \mathfrak{g}_N.\]
\vspace*{-15pt}
\subsection{$H^2(\R^4,\pi)$: Infinitesimal deformations} 

Combining Theorem \ref{main theorem} with the isomorphism from Equation \eqref{eq: iso cohom}, we obtain the following description of the formal Poisson cohomology in degree $2$. Here (and in the rest of this section) we use the following notation:
\[ W_i := \star^{-1}(\d \zeta_i).\]

\begin{corollary}\label{corollary: second PCH}
$H^2(\R^4,\pi)$ has unique representatives of the form
\begin{equation*}
    p E_1\wedge E_2 +q\pi + \sum_{i=1}^2 p_iE_i\wedge T_1 + q_iW_i+\star^{-1}\left(\sum_{i=1}^4\d(a_ix_i)\wedge \d f_1\right)
\end{equation*}
where $p,p_i,q,q_i \in  \R[[f_1,f_2]]$ and $a_i \in \R[[x_2^2,x_4]]$.
\end{corollary}
This follows from the identities
\begin{align}  
    \star^{-1} (\zeta_1\wedge \zeta_2) &\, = -E_1\wedge E_2 \label{eq: rel e12 zeta12}\\
    \star^{-1}(\d f_1\wedge \zeta_1) = -4E_i\wedge T_i +f_iW_i  \ &\text{ and }\ 
    \star^{-1}(\d f_1\wedge \zeta_2)= 4E_1\wedge T_2 +f_1W_2 = -4E_2\wedge T_1+f_2W_1\nonumber
\end{align} 
To study this result in more detail we use the filtration on $\mathfrak{X}^{\bullet}_f$ induced by $f$. We set
\begin{equation}\label{eq:filtration}
    \mathfrak{X}^{\bullet}_{f,1}:= \ker \iota_{\d f_1\wedge \d f_2} \cap\mathfrak{X}^\bullet_{f} \quad \text{ and }\quad \mathfrak{X}^\bullet _{f,2}:= \ker \iota_{\d f_1}\cap \ker \iota_{\d f_2} \cap\mathfrak{X}^\bullet_{f}
\end{equation}
Note that, since $f_1$ and $f_2$ are Casimir functions, $\d_{\pi}$ preserves $\mathfrak{X}^\bullet_{f,1}$ and $ \mathfrak{X}^\bullet_{f,2}$, respectively.
Infinitesimal deformations which preserves $\F$ are governed by cohomology classes represented by elements in $\mathfrak{X}^2_{f,2}$ up to coboundaries. It follows from the proof of the degree $2$ part of Theorem \ref{main theorem}, together with the identity
\begin{align}\label{eq: contraction vs wedge}
    \iota_{\alpha}(\star ^{-1} \beta)= (-1)^{k(4-k)}\star^{-1} (\beta\wedge \alpha)
\end{align} 
for $\alpha \in \Omega^k, \beta\in \Omega^l$ and $0\le k\le n-l$, that the only such cohomology classes are of the form $q\pi$ for $q\in \R[[f_1,f_2]]$. All such classes can be realized by a formal Poisson deformation of $\pi$ of the form
\[ \pi_t:= (1+tq)\pi \]
More generally, from \eqref{eq: Poisson via fibration} we know that any deformation of the volume form $\mu$ induces a deformation of Poisson structures preserving the foliation. By Corollary \ref{corollary: deformations} we conclude that any Poisson structure obtained in this way is formally equivalent to one as described above. 

On the other end of the spectrum of deformation we have those which induce a symplectic structure almost everywhere. There are two different classes of such deformations. First, the deformations represented by elements in $\mathfrak{X}^{2}_{f}/\mathfrak{X}^{2}_{f,1}$. By Corollary \ref{corollary: second PCH} such deformations are infinitesimally described by the classes $pE_1\wedge E_2$. Indeed, any bivector 
\[ \pi_t:= \pi +tpE_1\wedge E_2\]
is a Poisson bivector whose rank is $4$ away from the origin. This can be easily seen from the identities
\[ \pi = -8\, T_1\wedge T_2 \quad \text{ and }\quad 16 E_1\wedge E_2\wedge T_1\wedge T_2 = (f_1^2+f_2^2)\partial_1\wedge \partial_2\wedge \partial_3\wedge \partial_4,\]
 the fact that $E_i$ and $T_j$ commute and since $T_j\in \mathfrak{X}^{1}_{f,2}$ by \eqref{eq: relation E, T and f}. The other deformations we want to consider here are of the form
 \[ \pi_{1,t} = \pi +tW_1 \quad \text{ and }\quad \pi_{2,t}=\pi+tW_2.\]
 These bivectors are Poisson and we obtain a symplectic structure on $\R^4$ for $t>0$ since
 \[ \pi\wedge W_i = 0 \quad \text{ and }\quad W_i\wedge W_i= 2\partial_1\wedge \partial_2\wedge \partial_3\wedge \partial_4.\]

\subsection{Higher degrees and the algebraic structure}
Next we consider deformations of the form
\begin{equation}\label{eq: deformation}
  \pi_{i,t} = \pi + tp W_i.  
\end{equation}
To this end we first need to compute $H^3(\R^4,\pi)$ and $H^4(\R^4,\pi)$.

From Theorem \ref{main theorem} we obtain the following:
\begin{corollary}
The formal Poisson cohomology of $\pi$ satisfy:
    \begin{itemize}
    \item the group $H^3(\R^4,\pi)$ has unique representatives of the form
    \begin{equation*}
        \sum_{j=1}^2 p_j E_j\wedge W_1 + q_j T_j\wedge W_1 + \sum_{i=1}^4 b_ix_i T_1\wedge W_1 +\star^{-1}\d(a_ix_i)
    \end{equation*}
    where $p_j,q_j \in  \R[[f_1,f_2]]$ and  $a_i,b_i \in \R[[x_2^2,x_4]]$;
    \item the group $H^4(\R^4,\pi)$ has unique representatives of the form
    \begin{equation*}
        \left(p +  \sum_{i=1}^4 a_ix_i\right) \partial_1\wedge \partial_2\wedge \partial_3\wedge \partial_4
    \end{equation*}
    where $p \in  \R[[f_1,f_2]]$, $a_i \in \R[[x_2^2,x_4]]$ and where we write $\partial_j=\partial_{x_j}$.
    \end{itemize}
\end{corollary}
The statement follows from the following identities
\begin{align*}
    \star^{-1} \zeta_1 = E_1\wedge W_1=-E_2\wedge W_2 \quad \text{ and }\quad \star^{-1} \zeta_2 = E_1\wedge W_2=E_2\wedge W_1\\
    \star^{-1} \d f_1 = -4 T_1\wedge W_1=4T_2\wedge W_2 \quad \text{ and }\quad \star^{-1} \d f_2 = -4T_1\wedge W_2=-4T_2\wedge W_1
\end{align*}
Going back to the deformations from \eqref{eq: deformation} we obtain, for $i=1$ that:
\begin{align*}
    [\pi_{1,t},\pi_{1,t}]= 8t^2 p (\partial_xp T_1\wedge W_1 +\partial_yp T_2\wedge W_1) 
\end{align*}
In order to make the bivector Poisson, we would need to add a second order element in $t$, i.e. consider a deformation of the form
\[ \pi_{1,t} = \pi +tpW_1 + t^2 W\]
for some $W\in \mathfrak{X}^2_f$ such that
\[ [\pi ,W] = 4 p (\partial_xp T_1\wedge W_1 +\partial_yp T_2\wedge W_1).\]
However, the right side represents a non-trivial class in $H^3(\R^4,\pi)$ iff $p\in \R[[f_1,f_2]]$ is not constant and hence it is not possible to find such a $W$ for non-constant $p$. The same argument shows that $\pi_{2,t}$ does not define a deformation of Poisson structures.

\subsubsection*{The algebraic structure in cohomology}
The cohomology groups $H^{\bullet}(\R^4,\pi)$ carry a rich algebraic structure, as both the wedge product and the Schouten bracket
\begin{align*}
    \wedge: \mathfrak{X}^{k}\times \mathfrak{X}^{l} \to \mathfrak{X}^{k+l}\quad \text{ and }\quad [\cdot,\cdot ]: \mathfrak{X}^{k}\times \mathfrak{X}^{l} \to \mathfrak{X}^{k+l-1}
\end{align*}
descend to cohomology, inducing a Gerstenhaber algebra structure for Poisson cohomology:
\[\left(H^{\bullet}(\R^4,\pi),\wedge,[\cdot,\cdot]\right). \]
We distinguish two classes of representatives, those with coefficients in $\R[[f_1,f_2]]$ and those with coefficients in $\R[[x_2^2,x_4]]$. For the former, we note that the wedge product and the bracket can be easily deduced from the various relations described above. Describing the precise algebraic structure of representatives with coefficients in $\R[[x_2^2,x_4]]$ is harder and we do not attempt to describe them here. We only want to point out that the corresponding modules, as modules over $H^0(\R^4,\pi)$ are not free. 

Explicitly, we have:
\begin{proposition}\label{proposition: module structure}
For cohomology classes represented by some $a_i,b_i \in \R$ we have:
\begin{equation*}
    f_1a_ix_i = -2(x_2^2+x_4^2)a_ix_i \qquad \text{ and }\qquad f_1b_ix_i = -2(x_2^2+x_4^2)b_ix_i
\end{equation*}
in terms of the representatives in Theorem \ref{main theorem}. Moreover, we have the relations
\begin{equation*}
    a(f_1x_1+f_2x_2)=a(f_2x_1-x_2f_1)=a(f_1x_3+f_2x_4)=a(f_2x_3-f_1x_4)=0
\end{equation*}
in cohomology for $a\in \R$ and similarly for the classes represented by $b_i$.
\end{proposition}
\begin{proof}
    \noindent\underline{Degree 4:} We note that
\begin{align*}
      x_1f_1\mu = -2(x_2^2+x_4^2)x_1\mu +\dif f_1\wedge \dif f_2\wedge \nu_1
\end{align*}
where
\begin{align*}
    \nu_1 := \frac{1}{4}\big(x_3\dif x_2\wedge \dif x_3-x_4\dif x_1\wedge \dif x_3 +x_1\dif x_3\wedge \dif x_4\big)
\end{align*}
which is closed and hence there exists a primitive $\alpha_1\in \Omega^1(\SR)$. Similarly we obtain
\begin{align*}
      x_1f_2\mu = -2(x_2^2+x_4^2)x_2\mu +\dif f_1\wedge \dif f_2\wedge \nu_2
\end{align*}
for the closed two form
\begin{align*}
    \nu_2 := \frac{1}{2}\big(-2x_4\dif x_1\wedge \dif x_4 -x_4\dif x_1\wedge \dif x_3 +x_2\dif x_3\wedge \dif x_4\big).
\end{align*}
Moreover, we have 
\begin{align*}
    (x_1f_1+x_2f_2)\mu = \dif f_1\wedge \dif f_2\wedge \nu
\end{align*}
for the closed two-form
\begin{align*}
    \nu:= \frac{1}{4}\big(x_4\dif x_2\wedge \dif x_4-x_3\dif x_1\wedge \dif x_4 +x_1\dif x_3\wedge \dif x_4\big)
\end{align*}
The other cases follow along the same lines.

\noindent\underline{Degree 2:} We make use of the primitives in degree $4$.  Consider the $2$-cocylce given by
\begin{align*}
    \star W  =\d(x_1f_1+2(x_2^2+x_4^2)x_1)\wedge \d f_2.
\end{align*}
Setting $\star X:= -\d f_1\wedge \nu_1$ we obtain that
\begin{equation*}
    W=\d_{\pi} X.
\end{equation*}
Next consider the $2$-cocylce given by
\begin{align*}
    \star W  =\d(x_1f_1+x_2f_2)\dif f_2.
\end{align*}
If we set $\star X:= -\d f_1\wedge \nu$ then we obtain that
\begin{equation*}
    W=\d_{\pi} X.
\end{equation*}
The other cases follow similarly using the elements from the proof in degree $4$.

\noindent\underline{Degree 3:} Consider the $3$-cocylce given by
\begin{align*}
    \star T  =(x_1f_1+2(x_2^2+x_4^2)x_1)\dif f_1.
\end{align*}
Let $\alpha_1$ be a primitiv of $\nu_1$. Then we obtain for $\star W:= -\d f_1\wedge \alpha_1 $ that
\begin{equation*}
    T=\d_{\pi} W.
\end{equation*}
Next consider the $3$-cocylce given by
\begin{align*}
    \star T  =(x_1f_1+x_2f_2)\dif f_1.
\end{align*}
Let $\alpha$ be a primitiv of a closed $2$-form $\nu$ and define $\star W:= -\d f_1\wedge \alpha$. Then we obtain that
\begin{equation*}
    T=\d_{\pi} W.
\end{equation*}
Similarly we can treat the cocycle
\begin{align*}
    \star T  =\d (x_1f_1+2(x_2^2+x_4^2)x_1).
\end{align*}
Defining $\star W:= -\nu_1$ we obtain
\begin{equation*}
    T=\d_{\pi} W.
\end{equation*}
The other cases follow similarly using the primitives from the proof in degree $4$.
\end{proof}

\section{Preliminaries}\label{sec:preliminaries}
In this section we recall some facts from Poisson geometry and (homological) algebra. The first section is dedicated to Poisson geometry. In particular, we recall the definition of Jacobi-Poisson structures and describe the Poisson structure we are studying. In the second part we recall several notions from algebra: Hilbert-Poincare series, regular sequences and standard bases.

\subsection{Jacobi Poisson structures}\label{sec:prel poisson}
We recall the notion of Jacobi-Poisson structures as introduce in \cite{Dam89} attributed to Flaschka and Ratiu. See also \cite{GMP93}. In particular, we give an explicit formula for $\pi$ from Definition \ref{def: poisson biv} and we compute the Poisson differential $\dif_{\pi}$ in the various degrees.

Given a vector space $\R^n$ with standard volume form $\mu$ and $n-2$ functions $f_1, \dots f_{n-2}\in C^{\infty}(\R^n)$ one can define a Poisson structure $\pi$ on $\R^n$ by the following relation for the associated Poisson bracket:
\begin{equation*}
    \{g,h\}\mu := \dif g\wedge\dif h\wedge \dif f_1\wedge\dots \wedge \dif f_{n-2} \qquad \text{ for } g,h\in C^{\infty}(\R^n).
\end{equation*}
Poisson structures of this form are called \emph{Jacobi-Poisson} structures. Note that the leaves of such Poisson structures have dimension at most $2$. For a generalization with leaves of higher dimension see \cite{DP12}. For some properties of Jacobi-Poisson structures and a generalization to manifolds we refer to \cite{Prz01} and \cite{GSV14}. One property is the following:
\begin{lemma}\label{lemma: JP uni}
Any Jacobi-Poisson structure $\pi\in \mathfrak{X}^2(\R^n)$ is unimodular.
\end{lemma}
\begin{proof}
We observe that using \eqref{eq: isomorphism form vf} we have the identity
\begin{align*}
   \star (\pi) = \d f_1 \wedge \dots \wedge \d f_{n-2}.
\end{align*}
Hence using the definition of the modular vector field the result follows.
\end{proof}
As a consequence of the Lemma together with \eqref{eq:partialpi_identity} we obtain an isomorphism of complexes:
\begin{align*}
    \star: (\mathfrak{X}^{\bullet}(\R^n),\dif_{\pi})\diffto(\Omega^{n-\bullet}(\R^n),\delta_\pi).
\end{align*}

\subsubsection*{Poisson structure associated to a Lefschetz fibration}
Consider $f:\R^4\to \R^2$ as in \eqref{eq: lf} together with
the standard volume forms $\mu_4 $ and $\mu_2$ on $\R^4$ and $\R^2$ respectively, given by
\begin{align*}
    \mu_4=\dif x_1\wedge \dif x_2\wedge \dif x_3\wedge\dif x_4 \qquad \text{ and }\qquad \mu_2=\dif y_1\wedge \dif y_2.
\end{align*} 
We often write $\mu$ instead of $\mu_4$. Then $\pi$ from Definition \ref{def: poisson biv} is a Jacobi-Poisson structure and satisfies
\[ \star(\pi)=\dif f_1\wedge \dif f_2.\]
Explicitly, the Poisson bivector $\pi$ is given by
\begin{align*}
    \frac{1}{4}\pi =&\ \ \ \ (x_1^2+x_2^2)\partial _3\wedge \partial_4 - (x_1x_4-x_2x_3)(\partial_1\wedge \partial _3+\partial_2\wedge \partial_4)\\
    &+(x_3^2+x_4^2)\partial_1\wedge \partial_2 + (x_1x_3+x_2x_4)(\partial_2\wedge \partial_3-\partial_1\wedge \partial_4)
\end{align*} 
As an immediate consequence we can describe the Poisson differential $\delta_{\pi}$ on differential forms.
\begin{proposition}\label{proposition: Poisson differential}
The Poisson differential $\delta_{\pi}$ is given as follows.
\begin{itemize}
    \item In degree 1 we get for $\alpha\in \Omega^1(\R^4)$ that
    \begin{align*}
        \delta_{\pi}(\alpha) =& \ \iota_{\pi}(\dif \alpha)
    \end{align*}
    \item For a $2$-form $\beta\in \Omega^2(\R^4)$ we obtain the formula
    \begin{align*}
        \delta_{\pi}(\beta) =&\iota_{\star^{-1}\dif\beta}(\dif f_1\wedge \dif f_2) -\dif \iota_{\pi}(\beta)  
    \end{align*}
    \item In degree $3$, the differential of $\gamma\in \Omega^3(\R^4)$ is given by
    \begin{align*}
        \delta_{\pi}\gamma =\iota_{\pi} (\d \gamma) - \d \iota_{\star^{-1}\gamma}(\d f_1 \wedge\d f_2)
    \end{align*}
    \item For the top degree we obtain for any $g\in C^{\infty}(\R^4)$ that
    \begin{align*}
        \delta_{\pi}(g\mu) =  \dif g\wedge \dif f_1\wedge \dif f_2
    \end{align*}
\end{itemize}
\end{proposition}
The proof follows from the definition of $\delta_\pi$ and the identity
\[ \iota_{V}(\alpha)= (-1)^{k(n-l)}\iota_{\star^{-1}\alpha}(\star V),\]
for $V\in \mathfrak{X}^k(\R^n)$ and $\alpha \in \Omega^l(\R^m)$ with $k\le l\le m$.

\begin{remark}
One can generalize the formulas for $\delta_{\pi}$ in Proposition \ref{proposition: Poisson differential} to describe the differentials associated to any Jacobi-Poisson structure.
\end{remark}

\subsection{(Homological) Algebra}
\label{sec:prel algebra}This section intends to provide the algebraic background needed for the paper. We first recall the definition of the Hilbert-Poincare series together with an example which we will use later on. After a brief recap of regular sequences in the second section, we recall the notion of standard bases in the third part. As an application, we provide a couple of examples which will be used in later computations.

\subsubsection{Hilbert-Poincare series}\label{sec:prel hilbertseries} A reference for the background material in this sequence is \cite{AM69}.

Let $\SR$ be a commutative ring and $C$ a class of $\SR$-modules. A function $\lambda:C\to \Z$ is called \emph{additive} if for every short exact sequence of $\SR$-modules in $C$:
\begin{equation*}
    0\to M'\to M \to M'' \to 0 \qquad \Rightarrow \qquad \lambda (M)= \lambda(M')+\lambda(M'').
\end{equation*}
\begin{example}\label{ex: additive function}
For $\SR=\R$ and $C$ the class of all finite dimensional $\R$-vector spaces, $\lambda =\dim $ is a additive function.
\end{example}
\begin{proposition}\label{proposition: exact sequence additive}
Assume we have an exact sequence of $\SR$-modules in $C$
\begin{equation*}
    0\to M_0 \to M_1\to \dots \to M_n\to 0
\end{equation*}
such that all kernels of the homomorphisms are in $C$, then for any additive function $\lambda$ on $C$:
\begin{equation*}
    \sum_{i=0}^n (-1)^i\lambda(M_i)=0.
\end{equation*}
\end{proposition}
Let $\SR=\oplus_{n=0}^\infty\SR_n$ be a Noetherian graded ring. Note that $\SR_0$ is Noetherian and let $y_1,\dots ,y_s$ be the generators of $\SR$ as an $\SR_0$-algebra, of degrees $k_1,\dots ,k_r$, respectively. Let $\SM=\oplus_{n=0}^\infty\SM_n$ be a finitely-generated, graded $\SR$-module. Note in particular that this implies that all $\SM_i$ are finitely generated $\SR_0$-modules. Moreover, let $\lambda$ be an additive function on the class of all finitely-generated $\SR_0$-modules.
The \emph{Hilbert-Poincare series} of $M$ with respect to $\lambda$ is the power series
\begin{equation*}
    HP(\SM,t):= \sum_{i=0}^{\infty} \lambda(\SM_i) t^i.
\end{equation*}
In general the Hilbert-Poincare series can be described as follows:
\begin{theorem}[{\cite[Thm 11.1]{AM69}}]
$HP(\SM,t)$ is a rational function in $t$ of the form 
\begin{equation*}
    \frac{p(t)}{\Pi_{j=1}^r(1-t^{k_j})} \qquad \text{ for }\ p(t)\in \Z[t].
\end{equation*}
\end{theorem}
For any graded module $\SM$ we denote by $\SM(-a)$ the module with the degree shifted by $a\in\N$, i.e.
\begin{equation*}
    \SM(-a)_d = \begin{cases}\SM_{d-a} & \text{ if }\ 0\le d-a\\
    0&\text{ else.}
    \end{cases}
\end{equation*}
Hence by definition of the Hilbert-Poincare series we obtain for a degree shift the relation
\begin{equation*}
    HP(\SM(-a),t)=t^{-a}HP(\SM,t)
\end{equation*}
Another consequence for the Hilbert-Poincare series using Proposition \ref{proposition: exact sequence additive} is the following.
\begin{corollary}
Let $\SR^1, \dots ,\SR^n$ be Noetherian graded rings with $\SR^1_0 = \dots =\SR^n_0 = \SR_0$ which are all finitely generated as an $\SR_0$-algebra and let $\SM^1,\dots ,\SM^n$ be finitely-generated, graded $\SR^1,\dots , \SR^n$-modules respectively, fitting into an exact sequence
\begin{equation*}
    0\to \SM^1 \to \SM^2\to \dots \to \SM^n\to 0
\end{equation*}
such that all the homomorphisms preserve the grading and all kernels of the homomorphisms in the various degrees are in $C$. For $\lambda=\mathrm{dim}$ as in Example \ref{ex: additive function} we obtain using Proposition \ref{proposition: exact sequence additive} that
\begin{equation*}
    \sum_{i=0}^n (-1)^iHP(\SM^i,t)=0.
\end{equation*}
\end{corollary}
An example which we will also use later on is that of formal forms on $\R^n$:
\begin{example}\label{ex: HP formal forms}
Denote by $\SR:= \R[[x_1,\dots ,x_n]]$ the ring of formal power series in $n$-variables and by $\Omega^{\bullet}(\SR)$ the DGA of differental forms with coefficient in $\SR$, i.e.
\begin{align*}
    \Omega^{\bullet}(\SR)= \wedge^{\bullet}(\R^n)^*\otimes \SR
\end{align*}
The degree of $x_i$ is $1$ and the degree of $\dif x_i$ is $-1$ for all $1\le i\le n$. Then we obtain
\begin{align*}
    HP_{\Omega^m(\SR)}(t)=\binom{n}{m}\frac{t^m}{(1-t)^n}
\end{align*}
\end{example}

\subsubsection{Regular sequences}\label{sec:prel regularsequences}

Here we follow \cite{Pel09}. Let $\SR$ be a commutative ring and $\SM$ be a finitely generated $\SR$-module. We recall that we call an element $r\in \SR$ a \emph{non-zero divisor} on $\SM$ if
\begin{equation*}
    \forall m\in \SM: \ rm=0 \quad \Rightarrow m=0,
\end{equation*}
and a \emph{$\SM$-regular sequence} is a sequence $r_1,\dots , r_q\in \SR$ such that $r_i$ is a non-zero-divisor in the quotient
\begin{equation*}
    \SM / \langle r_1,\dots ,r_{i-1}\rangle \cdot \SM,
\end{equation*}
where $\langle r_1,\dots ,r_{i-1}\rangle \subset \SR$ denotes the ideal generated by $r_j$ for $1\le j\le i-1$. If $\SM=\SR$ then we simply call it a regular sequence.

\begin{example}\label{ex: reg sequence polynomial}
    The polynomials $x_1x_3+x_2x_4$ and $x_1x_4-x_2x_3$ form a regular sequence in $\R[[x_1,x_2,x_3,x_4]]$.
\end{example}
\begin{lemma}[see \cite{Wei94}]
Let $r_1,\dots , r_n$ be a regular sequence in $\SR$, then the Koszul complex
\begin{equation*}
    0\to \wedge^0 (\SR^n) \xrightarrow{\wedge \alpha}\dots \xrightarrow{\wedge \alpha} \wedge^n(\SR^n)\rightarrow \SR /\langle r_1 ,\dots , r_n\rangle \SR \to 0
\end{equation*}
is exact, where $\alpha=\sum r_ie_i$ and $(e_1,\dots ,e_n)$ is a free basis of $\SR^n$.
\end{lemma}

\begin{corollary}\label{lemma: exactness wedge f}
The operations $\wedge \d f_1$ and $\wedge \d f_2$ are exact in the sense of the above lemma.
\end{corollary}

\subsubsection{Standard bases}\label{sec:prel standardbases}
The notion of standard bases was introduced by Hironka in \cite{Hir64} and Buchberger \cite{Buc65}. For more details see e.g\, \cite{Eis95}[chapter 15] or \cite{GP08}[chapter 1]. We only give a brief summary of the most relevant facts for us. We will use standard bases to compute ideal membership and intersection of ideals in the ring of power series in an easy way. 

\begin{remark}
    All our computations are done for the various degrees of homogeneous polynomials. Hence it would be enough to use Gr\"obner basis. However, due to our general setting we decide to recap the theory for power series.
\end{remark}

Throughout this subsection let $\SR:= \R[[x_1,\dots ,x_n]]$. On $\SR$ we consider the local ordering obtained from the following monomial ordering (see \cite{GP08}[section 1.2]):
\[ x^{\alpha} > x^{\beta} \quad \:\Leftrightarrow\quad |\alpha|<|\beta| \ \text{ or }\ |\alpha|=|\beta| \text{ and } \exists 1\le i\le n: \ \alpha_1=\beta_1, \dots ,\alpha_{i-1}=\beta_{i-1},\alpha_i>\beta_i.\]
For $f\in \SR$ written as
\[ f= \sum_{\alpha\in \N^n} a_{\alpha}x^{\alpha}\] 
and any subset $G\subset \SR$ we define 
\[ LM(f):= \max \{ x^{\alpha}| a_{\alpha}\ne 0\}, \quad LC(f)=\{ a_{\alpha}| LM(f)=x^{\alpha}\} \quad \text{ and }\quad  L(G):= \langle LM(g)| g\in G\setminus\{0\}\rangle .\]
A standard basis is defined as follows.
\begin{definition}
    Let $\SI\subset \SR$ be an ideal. A standard basis of $\SI$ is a finite set $G\subset\SI$ with 
    \[ G\subset \SI \qquad \text{ and }\qquad L(\SI)=L(G).\]
\end{definition}
Moreover we need the following definitions.
\begin{definition}Let $G\subset \SR$.
\begin{enumerate} 
    \item $f\subset \SR$ is called reduced with respect to $G$ if no monomial of $f$ is contained in $LG(G)$. 
    \item $G$ is called reduced if $0\ne G$, for all $f\ne g\in G$ then $LM(f)\nmid LM(g)$ and for all $g\in G$, $LC(g)=1$ and $g-LM(g)$ is reduced in $G$.
\end{enumerate}
\end{definition}
Now we have the following proposition.
\begin{proposition}
Let $f ,f_1,\dots , f_m \in \SR$ then there exist $q_1,\dots,q_m,r\in \SR$ such that
\[ f=\sum_{i=1}^mq_if_i+r\]
and $r$ is reduced with respect to $\{f_1,\dots ,f_m\}$ and for all $i=1,\dots m$ we have $LM(q_if_i)\le LM(f)$.

Moreover, if $\{f_1,\dots , f_m\}$ is reduced, then $r$ is unique.
\end{proposition}
\begin{proof}
The existence of $r$ follows from \cite{GP08}[Theorem 6.4.1]. Uniqueness follows by comparing the monomials of two different $r$'s.    
\end{proof}
For $f_1,\dots f_m\in \SR$ as above we call $r$ the normal form of $f\in \SR$, i.e.
\[ NF(f|\{f_1,\dots f_m\}):=r.\]
\begin{corollary}\label{cor: ideal member}
    Let $\SI\subset \SR$ be an ideal, $G\subset \SI$ a reduced standard basis of $\SI$. Then for any $f\in \SR$:
    \[ f\in \SI \qquad \Leftrightarrow \qquad NF(f,G)=0.\]
\end{corollary}
\begin{proof}
    For a proof see \cite{GP08}[Lemma 1.6.7].
\end{proof} 
The following applications will be useful later. Let $\SR:= \R[[x_1,x_2,x_3,x_4]]$ and let $\SJ\subset \SR$ be the ideal defined by 
\begin{align}\label{eq: ideal df12}
    \SJ =\SJ(\d f_1,\d f_2) = \langle G:=\{x_1^2+x_2^2,x_3^2+x_4^2,x_1x_3+x_2x_4,x_1x_4-x_2x_3\}\rangle _\SR.
\end{align}
We have the following two results.
\begin{lemma}\label{ex: general 1}
Consider the module $\SM$ over $\R[[x_2^2,x_4]]$ defined by
\begin{align*}
    \SM:= \langle x_1,x_2,x_3,x_4\rangle_{\R[[x_2^2,x_4]]}.
\end{align*}
Then:
\begin{equation*}
    \SJ\cap \R[[f_1,f_2]]=\SJ\cap (\R[[f_1,f_2]]+\SM) = (f_1^2+f_2^2)\R[[f_1,f_2]] \qquad \text{ and } \qquad \R[[f_1,f_2]]\cap \SM=\{0\}.
\end{equation*}
\end{lemma}
\begin{proof}
Note that $G$ is reduced, hence we can use Corollary \ref{cor: ideal member}. Moreover, $\SJ$ is generated by homogeneous polynomials. Hence it is enough to check the statement for homogeneous degree polynomials. For odd homogeneous degree $n=2m+1$ we only have contributions from $\SM$. We write $h\in \SM$ as
\[ h= \sum_{j=1}^4 x_j\sum_{i=0}^m a_i^j x_2^{2i}x_4^{2(m-i)}.\]
Then we obtain by exchanging $x_1x_4$ with $x_2x_3$ that
\begin{align*}
    r:=NF(h|G)= &\, a_m^1 x_1x_2^{2m} +x_3\left(\sum_{i=0}^m a_i^3 x_2^{2i}x_4^{2(m-i)}+\sum_{i=0}^{m-1}a_i^1 x_2^{2i+1}x_4^{2(m-i)-1}\right) \\
    &\, +\sum_{i=0}^m \left(a_i^2x_2^{2i+1}x_4^{2(m-i)}+a_i^4x_2^{2i}x_4^{2(m-i)+1}\right).
\end{align*} 
Hence we have 
\[ h\in \SJ \quad \Leftrightarrow \quad r=0 \quad \Leftrightarrow \quad h=0.\]
For the even case $n=2m>0$ we observe that
\[ f_1^2 +f_2^2 =(x_1^2+x_2^2)^2 +(x_3^2+x_4^2)^2 +2(x_1x_3+x_2x_4)^2 -2(x_1x_4-x_2x_3)^2, \]
hence it is enough to consider $g\in \R[[f_1,f_2]]$ of the form
\[ g= g_1 f_1^m+g_2f_1^{m-1}f_2. \]
We have 
\begin{align}\label{eq: rest f1f2}
    r_g= NF(g|G)= g_1(-2)^{m}(x_2^2+x_4^2)^{m}+g_2(-1)^{m-1}2^m\left(x_1x_2^{2m-1}+x_3\frac{(x_2^2+x_4^2)^{m}-x_2^{2m}}{x_4}\right).
\end{align}

 We write an element $h\in \SM$ as
\[ h= \sum_{j=1}^4 x_j\sum_{i=0}^{m-1} a_i^j x_2^{2i}x_4^{2(m-i)-1}.\]
Note that
\begin{align*}
    r_h= &\, x_3\sum_{i=0}^{m-1} \left(a_i^3 x_2^{2i}x_4^{2(m-i)-1}+a_i^1 x_2^{2i+1}x_4^{2(m-i-1)}\right)  +\sum_{i=0}^{m-1} \left(a_i^2x_2^{2i+1}x_4^{2(m-i)-1}+a_i^4x_2^{2i}x_4^{2(m-i)+1}\right)
\end{align*}
and we conclude
\[ g+h\in \SJ \quad \Leftrightarrow \quad r_g+r_h=0 \quad \Leftrightarrow \quad g=h=0.\]
\end{proof}

\begin{lemma}\label{ex: general}
Let $g\in \R[[f_1,f_2]]$ and $(p,(q_1,q_2))\in \R[[x_2]] \oplus \R[[x_2,x_4]]^2$ be such that 
\begin{align*}
    g\cdot \mu =&\, \dif f_1 \wedge \dif (x_1 p+x_3q_1+q_2)(\dif x_1\wedge\dif x_4+\dif x_2\wedge\dif x_3) \mod \SJ\cdot \mu , 
\end{align*}
where $\mu$ is the standard volume form on $\R^4$. Then we have:
\begin{align*}
    x_2p+x_4q_1, q_2 \in \R[[x_2^2+x_4^2]] \qquad \text{ and }\qquad g\in (f_1^2+f_2^2)\R[[f_1,f_2]],
\end{align*}
and the same result holds if we replace $\dif f_1$ by $\dif f_2$.
\end{lemma}

\begin{proof}
By a direct computation we obtain that the right hand side of the statement equals:
\begin{align*}
    2\Big(-x_2 q_1  - x_3(x_1\partial_2 p+ x_3 \partial_2 q_1 + \partial_2 q_2) + x_1(x_3 \partial_4 q_1 + \partial_4 q_2)+x_4p\Big)\mu.
\end{align*}
Note that again it is enough to check the statement for homogeneous polynomials. We set
\[ p = p_0x_2^{n-1},\quad q_1 = \sum_{i = 0}^{n-1} a_i x_2^i x_4^{n-1-i},\quad q_2 = \sum_{i=0}^{n} b_i x_2^i x_4^{n-i}.\]
Then we obtain for the right hand side:
\begin{align*}
    r_R:= NF(& -x_2 q_1  - x_3(x_1\partial_2 p+ x_3 \partial_2 q_1 + \partial_2 q_2) + x_1(x_3 \partial_4 q_1 + \partial_4 q_2)+x_4p | G) =   \\
    &\, -\sum_{i= 0}^{n-1} a_i x_2^{i+1}x_4^{n-1-i}  +(n-1)p_0x_2^{n-1}x_4 +\sum_{i= 0}^{n-1} ia_i x_2^{i-1}x_4^{n-i+1} -\sum_{i= 0}^{n} ib_i x_2^{i-1} x_3 x_4^{n-i}  \\
    &\, -\sum_{i= 0}^{n-1}(n-i-1) a_i x_2^{i+1} x_4^{n-1-i} +  \sum_{i= 0}^{n-2}(n-i)b_i x_2^{i+1} x_3 x_4^{n-i-2} + b_{n-1}x_1 x_2^{n-1} +p_0x_2^{n-1}x_4\\
    = &\, np_0x_2^{n-1}x_4 +a_1x_4^n -a_{n-1}x_2^{n}-2a_{n-2}x_2^{n-1}x_4-\sum_{i= 0}^{n-3}\left((n-i) a_i -(i+2)a_{i+2}\right)x_2^{i+1}x_4^{n-1-i} \\
    & -b_1x_3x_4^{n-1}+ b_{n-1}x_1 x_2^{n-1}+ \sum_{i= 0}^{n-2}\left( (n-i)b_i - (i+2)b_{i+2} \right) x_2^{i+1} x_3 x_4^{n-i-2} .
\end{align*}
If $n=2m+1$ is odd then $g$ does not contribute and we obtain the condition
\begin{align*}
    (2m+1-i)a_i=-(i+2) a_{i+2}, \ 0\leq i \leq 2m-2,\quad a_1 =  a_{2m}= 0, \quad 2a_{2m-1}=(2m+1)p_0\\
    (i+2) b_{i+2} =(2m+1-i)b_i ,\  0 \leq i \leq 2m-1, \quad b_1 = b_{2m} = 0,
\end{align*}
which imply that $p=q_1=q_2=0$.

For $n=2m$ is even we note that
\[ f_1^2 +f_2^2 =(x_1^2+x_2^2)^2 +(x_3^2+x_4^2)^2 +2(x_1x_3+x_2x_4)^2 -2(x_1x_4-x_2x_3)^2. \]
Hence it is enough to consider $g = g_1 f_1^m + g_2 f_1^{m-1}f_2$ as in the previous lemma. From \eqref{eq: rest f1f2} we obtain the conditions
\begin{align*}
    2(m-i)a_{2i}=2(i+1) a_{2(i+1)}, \quad  
    (2(m-i)-1)a_{2i+1}-(2i+3) a_{2i+3}=(-2)^{m-1}\binom{m}{i+1}g_1, \ 
\end{align*}
for $0\leq i \leq m-2$ and
\begin{align*}
    a_1 = -(-2)^{m-1}g_1, \quad  a_{2m-1}=(-2)^{m-1}g_1 , \quad 2mp_0=2a_{2m-2},
\end{align*}
as well as
\begin{align*}
    2(i+1) b_{2(i+1)} &\, =2(m-i)b_{2i} ,\qquad  0 \leq i \leq m-1,\\
    (2i+3) b_{2i+3} &\, =(2(m-i)-1)b_{2i+1}+ (-2)^{m-1}\binom{m}{i+1}g_2 ,\quad  0 \leq i \leq m-2,\\
    b_1 = (-2)^{m-1}g_2 \quad &\, \quad   b_{2m-1} = -(-2)^{m-1}g_2,
\end{align*}
which implies the statement. The result for $\dif f_2$ follows along the same lines.
\end{proof}

\section{Division groups}\label{sec: division grous}

Consider the space $\Omega^\bullet(M)$ of differential forms on a smooth manifold $M$. It is possible to divide differential forms in the following sence: let $\alpha \in \Omega^1(M)$ be nowhere vanishing then for any $\beta \in \Omega^\bullet(M)$,
\[ \beta \wedge \alpha = 0 \quad \implies\quad \beta = \mu \wedge \alpha,\]
for some $\mu \in \Omega^{\bullet -1}(M)$. This property is extremely useful when manipulating forms but does not hold in general for the exterior algebra of a module. To measure the failure, de Rham \cite{Rha54} introduced the following groups.

Consider a free $\SR$-module $\SM$ with basis $(e_1,\dots,e_n)$. Fix elements $\alpha_1,\dots,\alpha_k \in \SM$. The quotient
\[ \SD^p(\alpha_1,\dots,\alpha_k) := \{ \beta\in \Lambda^p\SM\mid \beta \wedge\alpha_1 \wedge \dots \wedge \alpha_k = 0\}/ \sum_{i=1}^k\left(\alpha_i \wedge \Lambda^{p-1}\SM\right),\]
is called the \textbf{$p$-th division group} associated to $(\alpha_1,\dots,\alpha_k)$.

We make use of a result by Saito \cite{Sai76} about these groups. To state the results we first define the following. Use the basis elements to write
\[ \alpha_1 \wedge \dots \wedge \alpha_k = \sum_{1\leq i_1<\dots<i_k<\leq n} a_{i_1\cdots i_k} e_{i_1} \wedge \dots \wedge e_{i_k},\]
with $a_{i_1\cdots i_k} \in \SR$, and define $\SJ = \SJ(\alpha_1,\dots,\alpha_k) \subset \SR$, the ideal generated by the coefficients $a_{i_1,\cdots,i_k}$.

\begin{definition}
The $\emph{depth}$ $\mathrm{depth}_{\SR}(I,\SM)$ of an ideal $I\subset \SR$ on a finitely generated module $\SM$ is the supremum of the lengths of all $\SM$-regular sequences of elements of $I$.
\end{definition}
\begin{proposition}[\cite{Sai76}]\label{prop:DepthInequality}
The division groups satisfy:
\[ \SD^p(\alpha_1,\dots,\alpha_k) = 0,\]
for all $0 \leq p < \operatorname{depth}(\SJ,\SR)$.
\end{proposition}

Now, let us return to the setting of Lefschetz singularities. We introduce the notation:
\begin{align*}
    \beta_1 :=\dif \zeta_1= \dif x_1\wedge \dif x_3 -\dif x_2\wedge \dif x_4 \quad \text{ and }\quad  \beta_2:= \dif \zeta_2=\dif x_1\wedge \dif x_4+\dif x_2\wedge \dif x_3,
\end{align*}
where $\zeta_i$ were defined in \eqref{eq: dual rot}. We have the following result.
\begin{proposition}\label{proposition: iso division group}
The depth of $\SJ=\SJ(\d f_1,\d f_2)$ is $2$, and there is an isomorphism:
\begin{align*}
    I: \ \ \R\oplus \R[[x_2]] \oplus \R[[x_2,x_4]]^2\ &\diffto \qquad \mathcal{D}^2(\d f_1,\d f_2)\\
    (c ,p,(q_1,q_2))\qquad  &\, \mapsto [c\beta_1 +(px_1+q_1x_3+q_2)\beta_2]
\end{align*}
\end{proposition}

\subsection{Consequence of Proposition \ref{proposition: iso division group}}\label{section: mistake explanation}
In this subsection we explain the problem of the claimed computation for the formal Poisson cohomology for Lefschetz singularities in \cite{BV20}.

Let us first recall the definition of an isolated complete intersection singularity. Set $\SR=\R[[x_1,\dots ,x_n]]$ and consider polynomials $p_1,\dots,p_k\in \R[x_1,\dots,x_n]$ with zero constant term. Define $\SI=\SI(p_1,\dots, p_k)$ to be the ideal generated by
\begin{align*}
    \langle p_1,\dots,p_k, \det (\partial_{j_i} p_i)_{1\le i\le k}\rangle.
\end{align*}
\begin{definition}
$(p_1,\dots,p_k)$ has an isolated complete intersection singularity (ICIS) if $(p_1,\dots,p_k)$ is a regular sequence in $\SR$ and $\SR/\SI$ is a finite dimensional $\R$-vector space.
\end{definition}
Let us define the ideal $\SJ=\SJ(p_1,\dots, p_k)$ by its generators
\begin{align*}
    \SJ:=\langle \det(\partial_{j_i} p_i)_{1\le i\le k}\rangle_{\SR}.
\end{align*}
The following is due to \cite{Loo84}[Proof of Proposition (4.4)]:
\begin{proposition}\label{prop: icis depth}
If $(p_1,\dots,p_k)$ is an ICIS, then the depth of the ideal $\SJ$ in $\SR$ is $k+1$. 
\end{proposition}
For the computations in \cite{BV20} it is assumed that $(f_1,f_2)$ is an ICIS (p. 16, last paragraph). This assumption is crucial for the computations as done in \cite{Pel09}. However, combining \ref{proposition: iso division group} and Proposition \ref{prop: icis depth} implies that $(f_1,f_2)$ is not an ICIS. In fact, as we shall see, the non-vanishing of $\mathcal{D}^2$ seems to be main factor for some of the cohomology groups being non-free modules over $H^0$.

\subsection{Proposition \ref{proposition: iso division group}: The proof and a Corollary}

We start with the proof of Proposition \ref{proposition: iso division group}. To see that $\operatorname{depth}(\SJ,\SR)= 2$ observe that
\begin{equation*}
    \beta_i \wedge \d f_1 \wedge \d f_2 = 0,\quad i =1,2.
\end{equation*}
Furthermore, the coefficients of $\beta_i$ are homogenenous of degree $0$, whereas the coefficients of $\d f_i$ are homogeneous of degree $1$. This implies that $\beta_1$ and $\beta_2$ define a non-zero classes in $\SD^2(\d f_1,\d f_2)$. Hence, by Proposition \ref{prop:DepthInequality} we have $\operatorname{depth}(\SJ,\SR) \le 2$. Hence Example \ref{ex: reg sequence polynomial} implies that $\operatorname{depth}(\SJ,\SR) = 2$.

We note first that the map is well-defined. Let $\omega \in \Omega^2(\SR)$ be of the form
\begin{align*}
    \omega=\sum_{i<j}\omega_{ij}\dif x_i\wedge \dif x_j \qquad \text{ with } \omega_{ij}\in \SR.
\end{align*}
The wedge product of $\omega$ with $\dif f_1\wedge \dif f_2$ is described by the function
\begin{align}\label{eq: omega wedge}
    \star^{-1}(\frac{1}{4}\omega\wedge \dif f_1\wedge \dif f_2)=& \quad \ (x_1^2+x_2^2)\omega_{34} - (x_1x_4-x_2x_3)(\omega_{13}+\omega_{24})\\
    &\, +(x_3^2+x_4^2)\omega_{12} + (x_1x_3+x_2x_4)(\omega_{23}-\omega_{14}). \nonumber
\end{align}
We note that for $\gamma \in \Omega^1(\SR)$ of the form
\begin{align*}
    \gamma = \sum_{i}\gamma_i \dif x_i
\end{align*}
we obtain that
\begin{align}
    \frac{1}{2}\dif f_1\wedge \gamma = & \ \ \ \ (x_1\gamma_2+x_2\gamma_1)\dif x_1\wedge \dif x_2 +(x_3\gamma_4+x_4\gamma_3)\dif x_3\wedge \dif x_4 \nonumber\\
    &+(x_1\gamma_3-x_3\gamma_1)\dif x_1\wedge \dif x_3+(x_4\gamma_2-x_2\gamma_4)\dif x_2\wedge \dif x_4\label{eq: primitiv 1}\\
    &+(x_1\gamma_4+x_4\gamma_1)\dif x_1\wedge \dif x_4-(x_2\gamma_3+x_3\gamma_2)\dif x_2\wedge \dif x_3.\nonumber
\end{align}
Similarly we have for $\delta\in\Omega^1(\SR)$ that 
\begin{align}
    \frac{1}{2}\dif f_2\wedge \delta = & \ \ \ \ (x_2\delta_2-x_1\delta_1)\dif x_1\wedge \dif x_2 +(x_4\delta_4-x_3\delta_3)\dif x_3\wedge \dif x_4 \nonumber\\
    &+(x_2\delta_3-x_4\delta_1)\dif x_1\wedge \dif x_3+(x_1\delta_4-x_3\delta_2)\dif x_2\wedge \dif x_4\label{eq: primitiv 2}\\
    &+(x_2\delta_ 4-x_3\delta_1)\dif x_1\wedge \dif x_4+(x_1\delta_3-x_4\delta_2)\dif x_2\wedge \dif x_3.\nonumber
\end{align}
We want to show that for any $\omega\in Z^2$ we can find $\gamma$ and $\delta$ such that
\begin{align}\label{eq: omega eq class}
    \bar{\omega}=\omega+\dif f_1\wedge \gamma+\dif f_2\wedge \delta
\end{align}
is of the form as described in Proposition \ref{proposition: iso division group}.

Let us assume that $\omega\in Z^2$ or equivalently that
\begin{align*}
    \star^{-1}(\omega\wedge \dif f_1\wedge \dif f_2) = 0.
\end{align*}
\textbf{Idea:} We show in steps how we can simplify (parts) of the coefficients of $\omega$ as follows:
\begin{enumerate}
    \item We make choices for (parts) of the coefficients of $\gamma_i$ in \eqref{eq: primitiv 1} and $\delta_j $ in \eqref{eq: primitiv 2} to show that these simplifications can be achieved.
    \item We continue in the next step with a choice for $\gamma_i$ and $\delta_j$ based upon the simplification for $\omega$ which we verified in the previous steps. 
\end{enumerate}

\textbf{Notation:} To keep notation simple:
\begin{itemize}
    \item We write $I\in \N_0^4$ with a subscript to indicate which entries vary over non-zero entries, e.g. $I_{24}$ stands for $(0,i_2,0,i_4)$ with $i_2,i_4\in \N_0$. We simply write $I$ for $I_{1234}$;
    \item For $r\in \SR$ the expression $r(I)$ refers to the coefficient of the term $x_1^{i_1}x_2^{i_2}x_3^{i_3}x_4^{i_4}$ for $r$; 
    \item If $J=(j_1,j_2,j_3,j_4)$ happens to be such that at least one of the $j_i$ is negative we set $r(J):=0$. 
\end{itemize}

We note first that the coefficients $\omega_{12}$ and $\omega_{34}$ satisfy:
\begin{align*}
    \omega_{12}(I_{34})=0 \qquad \text{ and }\qquad \omega_{34}(I_{12})=0.
\end{align*}
This follows immediately from \eqref{eq: omega wedge} since the only non-zero elements in the subrings $\R[[x_1,x_2]]$ and $\R[[x_3,x_4]]$ can come from the corresponding parts of $\omega_{34}$ and $\omega_{12}$, respectively.

\textbf{Claim 1:} We may assume that the coefficients $\omega_{12}$ and $\omega_{34}$ satisfy:
\begin{align}\label{eq: simp1}
    \omega_{12}=\omega_{34}=0.
\end{align}

We can choose
\begin{align}
    \delta_1(I):=&\, \frac{1}{2}\omega_{12}(I+e_1)+\gamma_2(I)+(\gamma_1+\delta_2)(I+e_1-e_2),\nonumber\\
    \delta_2(I_{234}):=&\, -\frac{1}{2}\omega_{12}(I_{234}+e_2)-\gamma_1(I_{234}),\qquad \delta_4(I_{124}):=-\frac{1}{2}\omega_{34}(I_{124}+e_4)-\gamma_3(I_{124})\label{eq: first claim division}\\
    \delta_3(I):=&\, \frac{1}{2}\omega_{34}(I+e_3)+\gamma_4(I)+(\gamma_3+\delta_4)(I+e_3-e_4).\nonumber
\end{align}
By comparing the coefficients in \eqref{eq: primitiv 1}, \eqref{eq: primitiv 2} and \eqref{eq: omega eq class} 
this implies the claim.

\textbf{Claim 2:} We may assume that
\begin{align*}
    \omega_{13}+\omega_{24}=0=\omega_{23}-\omega_{14} \quad \text{ i.e.\ }\quad r=0.
\end{align*}

By Example \ref{ex: reg sequence polynomial} the sequence $\{x_1x_3+x_2x_4, x_1x_4-x_2x_3\} \subset\SR$ is regular. Since $\SR$ is a unique factorization domain, the vanishing of \eqref{eq: omega wedge} together with \eqref{eq: simp1} implies the existence of $r\in \SR$ such that
\begin{align*}
    \omega_{13}+\omega_{24}=(x_1x_3+x_2x_4)r \quad \text{ and }\quad \omega_{23}-\omega_{14}=(x_1x_4-x_2x_3)r.
\end{align*}
We choose 
\begin{align*}
    \gamma_1(I+e_1):= (\gamma_3+\delta_4)(I+e_3)-\delta_2(I+e_1)-\frac{1}{2}r.
\end{align*}
Summing up the $\dif x_1\wedge \dif x_3$ and the $\dif x_2\wedge\dif x_4$-components of \eqref{eq: primitiv 1} and \eqref{eq: primitiv 2} yields
\begin{align*}
    (x_1x_3+x_2x_4)\big((\gamma_3+\delta_4)(I+e_3)-(\gamma_1+\delta_2)(I+e_1)\big).
\end{align*}
Similarly, we obtain for the $\dif x_1\wedge \dif x_4$-components minus the $\dif x_2\wedge\dif x_3$-components that
\begin{align*}
    -(x_1x_4-x_2x_3)\big((\gamma_3+\delta_4)(I+e_3)-(\gamma_1+\delta_2)(I+e_1)\big)
\end{align*}
implying the claim.

\textbf{Claim 3:} We may assume that
\begin{align*}
    \omega_{13}\in \R.
\end{align*}

Note that the sum of the $\dif x_1 \wedge \dif x_3$-components from \eqref{eq: primitiv 1} and \eqref{eq: primitiv 2} is given by
\begin{align*}
    x_1\gamma_3-x_3\gamma_1+x_2\delta_3-x_4\delta_1=x_1\gamma_3(I_{124})-x_3\gamma_1(I_{234})+x_2\gamma_4-x_4\gamma_2+x_1x_3(\delta_2(I+e_1)-\delta_4(I+e_3)).
\end{align*}
Hence by taking 
\begin{align}
    \gamma_2(I_4):=&\, \frac{1}{2}\omega_{13}(I_4+e_4), \qquad  \gamma_4(I_{24}):=-\frac{1}{2}\omega_{13}(I_{24}+e_2)+\gamma_2(I_{24}+e_2-e_4),\nonumber\\
    \gamma_3(I_{124}):=&\, -\frac{1}{2}\omega_{13}(I_{124}+e_1)-\gamma_4(I_{124}+e_1-e_2)+\gamma_2(I_{124}+e_1-e_4),\label{eq: third claim division}
    \\
    \gamma_1(I_{234}) :=&\,  \frac{1}{2} \omega_{13}(I_{234}+e_3)+\gamma_4(I_{234}+e_3-e_2)-\gamma_2(I_{234}+e_3-e_4),\nonumber\\
    \delta_4(I+e_3):=&\, \frac{1}{2}\omega_{13}(I+e_1+e_3)+\delta_2(I+e_1)+\gamma_4(I+(1,-1,1,0))-\gamma_2(I+(1,0,1,-1))\nonumber
\end{align}
we obtain the claim.

To simplify $\omega_{14}$ we add the $\dif x_1\wedge \dif x_4$-components of \eqref{eq: primitiv 1} and \eqref{eq: primitiv 2}, which yields
\begin{equation}\label{eq: last case division}
    x_1\gamma_4+x_4\gamma_1+x_2\delta_4-x_3\delta_1=\begin{array}{l}
    (x_1^2+x_2^2)\gamma_4(I_{124}+e_1)+(x_1x_3+x_2x_4)\gamma_4(I+e_3)\\
     -  (x_3^2+x_4^2)\gamma_2(I+e_3)+(x_1x_4-x_2x_3)(\gamma_3(I+e_3)+\gamma_2(I_{24}+e_2)).
    \end{array}
\end{equation}
Note that all the different coefficients are independent of each other. Moreover, we have
\begin{align}\label{eq: partition of sj}
    \SJ=  \SR\cdot(x_3^2+x_4^2) + \SR\cdot (x_1x_3+x_2x_4)+\SR \cdot (x_1x_4-x_2x_3)+\R[[x_1,x_2,x_4]]\cdot (x_1^2+x_2^2).
\end{align} 
Hence the result follows from:

\textbf{Claim 4:} The following map is an isomorphism of $\R$-vector spaces: 
\begin{equation*}
    \begin{array}{rccc}
         m:&  \R[[x_2]]\oplus \R[[x_2,x_4]]^2  &\diffto & \SR / \SJ \\
    &(p,(q_1,q_2)) &\mapsto    & [px_1+q_1x_3+q_2].
    \end{array}
\end{equation*}

We first show that the map $m$ is injective. In order to do this we use that the linear map:
\begin{align*}
     l: \R[[x_1,x_2]]\oplus\R[[x_3,x_4]] \oplus \SR^2  &\to \qquad \qquad \qquad \qquad \qquad \qquad \qquad \SJ \\
    (p_1,p_2,(r_1,r_2))\ \qquad &\mapsto    \, (x_1^2+x_2^2) p_1 +(x_3^2+x_4^2) p_2 + (x_1x_3+x_2x_4)r_1+(x_1x_4-x_2x_3) r_2 
\end{align*}
is surjective with kernel given by
\[ (0,0,((x_1x_4-x_2x_3) \cdot r,-(x_1x_3+x_2x_4)\cdot r)) \qquad \text{ for } r\in \SR. \]
This can be seen from Example \ref{ex: reg sequence polynomial} and the relations
\begin{align*}
    x_3(x_1^2+x_2^2) -x_1(x_1x_3+x_2x_4)+x_2(x_1x_4-x_2x_3)=0,&\\ x_4(x_1^2+x_2^2) -x_2(x_1x_3+x_2x_4)-x_1(x_1x_4-x_2x_3)=0,&\\
    x_1(x_3^2+x_4^2) -x_3(x_1x_3+x_2x_4)-x_4(x_1x_4-x_2x_3)=0,&\\
    x_2(x_3^2+x_4^2) -x_4(x_1x_3+x_2x_4)+x_3(x_1x_4-x_2x_3)=0.&
\end{align*}
Let $(p,(q_1,q_2))\in \R[[x_2]]\oplus \R[[x_2,x_4]]^2$ such that $m(p,(q_1,q_2))=0$, i.e.
\[\exists (p_1,p_2,(r_1,r_2))\in \R[[x_1,x_2]]\oplus\R[[x_3,x_4]]  \oplus \SR^2  : \quad l(p_1,p_2,(r_1,r_2))= px_1+q_1x_3+q_2. \]
Comparing the terms in $\R[[x_1,x_2]]$ and $\R[[x_3,x_4]]$ we immediately obtain
\[ p= 0 \qquad q_1(I_4)=0, \qquad q_2(I_2)=0 \qquad q_2(I_4)=0\quad \text{ and }\quad p_1=p_2=0.\]
Hence we can write the image of $m$ uniquely as 
\[ x_2x_3 \tilde{q}_1+ x_2x_4\tilde{q}_2 \qquad \text{ for some } \tilde{q}_1,\tilde{q}_2\in \R[[x_2,x_4]]. \]
But such an element is in the image of $l$ iff it is zero, hence $m$ is indeed injective.

We conclude the proof by comparing dimensions of the domain and codomain of $m$ for a fixed homogeneous degree in $\SR/\SJ$. We denote by $\SR_d$ and $\SJ_d$ the corresponding subvector spaces generated by polynomials of homogeneous degree $d\in \N_0$. Using the fact that for $n\in N$ and $d\in \N_0$ we have
\begin{align*}
    \dim \R[[x_1,\dots ,x_n ]]_d = \binom{n-1+d}{n-1},
\end{align*}
and the description of $\SJ$ via $l$ we obtain for the dimension of $\SJ_d$ for $2\le d$ that
\begin{align*}
    \text{ for } 2\le d<4:& \quad \dim \SJ_d =2\binom{d-1}{1}+2\binom{d+1}{3},\\
    \text{ and else :}& \quad  \dim \SJ_d =2\binom{d-1}{1}+2\binom{d+1}{3} - \binom{d-1}{3}.
\end{align*} 
Finally, a direct computation implies that for any $d\in \N$ we have
\[ \dim \SR_d/\SJ_d=\dim \SR_d - \dim \SJ_d= 2(d+1),\]
which proves the claim by comparing it with $\dim (m^{-1}(\{\SR_d/\SJ_d\}))$. Hence we proved Proposition \ref{proposition: iso division group}.
\begin{corollary}\label{lemma: relation second cohomology}
The isomorphism from Proposition \ref{proposition: iso division group} satisfies the relations:
\begin{align*}
I^{-1}( [f_1\cdot I(0,p,(q_1,q_2))])&\, = -2(0,x_2^2p,(x_2x_4p+(x_2^2+x_4^2)q_1,(x_2^2+x_4^2)q_2))\\   
I^{-1}( [f_2\cdot I(0,p,(q_1,q_2))])&\, =  2(0,q_2(I_2)x^{I_2+e_2}, (x_4q_2+x_2^2x_4^{-1}(q_2-q_2(I_2)x^{I_2}),-(x_2p+x_4q_1)(x_2^2+x_4^2)))
\end{align*}
for any $(p ,(q_1,q_2))\in \R[[x_2]]\oplus \R[[x_2,x_4]]^2$. Additionally, we have
\begin{align*}
    [f_1\beta_1-f_2\beta_2]&\, =0=[f_2\beta_1+f_1\beta_2].
\end{align*} 
In particular, for any $(p ,(q_1,q_2))\in \R[[x_2]]\oplus \R[[x_2,x_4]]^2$ which satisfy
\[ px_2 +x_4q_1 = a (x_2^2+x_4^2)^m \qquad \text{ and } \qquad q_2= b (x_2^2+x_4^2)^m\]
    for some $a,b\in \R$, we obtain that 
    \[
        [(x_1p + x_3q_1)\beta_2] = [c_1f_1^{m-1}f_2\beta_2]\quad \text{and} \quad [q_2\beta_2] = [c_2f_1^m\beta_2],
    \]
    for $c_1 := (-2)^{-m} a$ and $c_2:= 2^{-m}b$.
\end{corollary}
\begin{proof}
The first two equations follow by a direct computation using \eqref{eq: last case division} and its induced partition for $\SJ$ as stated in \eqref{eq: partition of sj}. The last two equalities are obtained from a direct computation and choices for $\gamma$ and $\delta$ as described in \eqref{eq: third claim division} and \eqref{eq: first claim division}.
\end{proof}

\section{Kernel of the Poisson differential}\label{sec: ker poisson diff}

The aim of this section is to compute the kernel of the Poisson differential $\delta_\pi:\Omega^\bullet_f \to \Omega^{\bullet -1}_f$. We prove the following:
\begin{proposition}\label{proposition: kernel differential}
    For the kernel of the differential $\delta_{\pi}$ from Proposition \ref{proposition: Poisson differential} we obtain the following results for formal differential forms. The statements are for $g,g_i\in \SR$ and $p,p_i,q_i\in \R[[f_1,f_2]]$.
    \begin{itemize}
        \item For $\alpha\in \Omega^1 _f$ we obtain that
        \begin{equation}\label{eq: ker degree 1}
            \alpha\in \ker \delta_{\pi} \quad \Leftrightarrow \quad  \alpha  =\dif g_0+\sum_{i=1}^2p_i\zeta_i +g_i\dif f_i.
        \end{equation}
        \item In degree $2$ we get for $\omega\in \Omega^2_f$ that
        \begin{align}\label{eq: ker degree 2}
            \omega\in \ker \delta_{\pi} \quad \Leftrightarrow \quad \omega= p\zeta_1\wedge\zeta_2+g\dif f_1\wedge\dif f_2 +\sum_{i=1}^2 \dif f_1\wedge p_i \zeta_i +\dif (q_{i}\zeta_i) +\dif g_i\wedge \dif f_i.
        \end{align}
        \item For a formal $3$-form $\gamma\in \Omega^3_f$ we obtain
        \begin{align}\label{eq: ker degree 3}
            \gamma \in \ker \delta_{\pi} \quad \Leftrightarrow \quad \gamma=&\  \dif f_1\wedge\dif f_2\wedge \dif g + \sum_{i=1}^2 q_i\epsilon_i +\dif f_1\wedge \dif  \big(p_i\zeta_i\big) 
        \end{align}
        where $\epsilon_i = \star E_i$.
        \item In the top degree we have for $g\in \SR$ that
        \begin{align}\label{eq: ker deg 4}
            g\mu\in \ker \delta_{\pi} \quad \Leftrightarrow \quad g\in \R[[f_1,f_2]].
        \end{align}
    \end{itemize}
\end{proposition} 
For top degree forms this is a consequence of \cite{Pel09}[Proposition 3.1], since $\mathcal{D}^1=0$. In the other degrees, Pelap uses the vanishing of $\mathcal{D}^2$, which doesn't hold in our case. Therefore, we need an adaptation which takes the non-triviality of $\mathcal{D}^2$ into account. The sections below provide proofs for the various degree and are ordered according to their logical dependence:
\[ \text{ degree 1 } \Rightarrow 
    \text{ degree 3} \Rightarrow 
    \text{ degree 2}.
\] 
Therefore, the proof for degree 1 is at the core of the proof of Proposition \ref{proposition: kernel differential}, which we obtain by an induction argument. Another important ingredient for the proof in any degree is a good understanding of the group $\mathcal{D}^2$, as described in Proposition \ref{proposition: iso division group} and Corollary \ref{lemma: relation second cohomology}. We have two more observations:
\begin{itemize}
    \item First we note that the standard scalar multiplication $m_t :\R^4\to \R^4$ on $\R^4$ satisfies
    \[ \delta_{\pi} \circ m_t^* = m_t^* \circ \delta_{\pi}. \]
    Hence $\delta_{\pi}$ maps homogeneous degree coefficients to homogeneous degree coefficients and it is enough to prove the statements for coefficients of a fixed homogeneous degree.
    \item It is easy to verify the implications $\Leftarrow$ in Proposition \ref{proposition: kernel differential} using Proposition \ref{proposition: Poisson differential}. Therefore, we really only need to show that $\Rightarrow$ holds. To do so it is enough to show that 
    \begin{equation}\label{eq: modulo rhs}
        \ker \delta_{\pi} \mod V = \{0\},
    \end{equation}
    where $V$ denotes the vector space generated by elements on the right hand side of Proposition \ref{proposition: kernel differential}. Whenever we look at this quotient we refer to it by $\operatorname{mod} V$. 
\end{itemize}
That being said, we are ready to start the proof of Proposition \ref{proposition: kernel differential}.

\subsection{Proof of Degree $1$}
For $\alpha\in \Omega^1_f$ recall from Proposition \ref{proposition: Poisson differential} that
\begin{align*}
     \delta_\pi(\alpha) &= \d \alpha \wedge \d f_1 \wedge \d f_2.
\end{align*}
Hence it suffices to prove that
\[ \d \alpha\wedge \d f_1 \wedge \d f_2  = 0 \quad \Leftrightarrow\quad \alpha = \d g_0 + \sum_{i=1}^2 p_i(f_1,f_2) \zeta_i + g_i \d f_i.\]
If the coefficients of $\alpha$ are of homogeneous degree $n+1$, i.e. $m_t^*\alpha = t^{n+2}\alpha$, the following lemma will allow us to derive the statement by induction over even and odd degrees, respectively.
\begin{lemma}\label{lemma: induction degree 3}
    For $n=0$ there exist $c_1 ,c_2\in \R$ such that
    \[ \d (\alpha -  (c_1\zeta_1+c_2\zeta_2) )= 0.\]
    For $ 1\le n=2m $ there exist $c_1 ,c_2\in \R$ such that
    \[ \d (\alpha-  (c_1f_1^m+c_2f_1^{m-1}f_2)\zeta_2) = \alpha_1 \wedge \d f_1 +\alpha _2\wedge \d f_2.
    \]
    Moreover, for $1\le n =2m+1$ we have 
    \[ \d \alpha = \alpha_1 \wedge \d f_1 +\alpha _2\wedge \d f_2.
    \]
    In all cases we may assume
    \[ m_t^*\alpha_i= t^n \alpha_i \qquad \text{ for } \quad i=1,2.\]
\end{lemma}
\begin{proof}
    From Proposition \ref{proposition: iso division group} we obtain that
    \[ \d \alpha\wedge \d f_1 \wedge \d f_2 = 0\quad \Rightarrow \quad  \d \alpha =I(c,p,(q_1,q_2)) +\sum_{i=1}^2\alpha_i \wedge \d f_i\]
    for some $c\in \R, p\in \R[[x_2]], q\in\R[[x_2,x_4]]^2$ and $\alpha_i\in \Omega^1_f$. Note that since $\alpha$ has coefficient of homogeneous degree $n+1$ we may assume that the coefficients of $I(c,p,q)$ and $\alpha_i$ are homogeneous of degree $n$ and $n-1$, respectively. Wedging the equation for $\d \alpha$ with $\d f_1$ and applying $\d$ yields
    \begin{align}\label{eq: tangent degree 1 general}
        \dif f_1 \wedge \dif I(c,p,q)=\dif f_1\wedge \dif f_2 \wedge \dif \alpha_2.
    \end{align}
    If $n=0$, then the statement follows immediately from the fact that $\dif \zeta_i =\beta_i$. Hence we assume $n>0$. Then \eqref{eq: tangent degree 1 general} implies the hypothesis of Lemma \ref{ex: general} for $g \cdot \mu = 0$, which in turn implies that $x_2 p + x_4q_1, q_2 \in R[[x_2^2 + x_4^2]]$. Hence Corollary \ref{lemma: relation second cohomology} implies $[(x_1p + x_3q_1)\beta_2] = [c_1f_1^{m-1}f_2\beta_2]$ and $[q_2\beta_2] = [c_2f_1^m \beta_2]$ so that 
    \[ I(0,p,q) = [(c_1f_1^{m-1}f_2 + c_2f_1^m)\beta_2] \in \mathcal{D}^2.\]
    Finally, since $\d \zeta_i = \beta_i$ we get
    \[ I(0,p,q)= [\d((c_1f_1^{m-1}f_2 + c_2 f_1^m)\zeta_2)] \in \mathcal{D}^2,\]
    and therefore by changing $\alpha_i$ we obtain the statement.    
\end{proof}
We prove the statement first for the case of $n=2m-1$ being odd. In the base case, i.e. $m=0$ it is straightforward to see that $\tau = \d g$ for $g$ of homogeneous degree $1$. Hence we assume
\[ m_t^*\alpha=t^{2m+3}\alpha \qquad \text{ and }\qquad \d f_1 \wedge \d f_2 \wedge \d \alpha = 0. \]
From Lemma \ref{lemma: induction degree 3} we get that
\[ \d\alpha =\sum_{i=1}^2\alpha_i \wedge \d f_i \quad \text{ with } \quad m_t^*\alpha_i =t^{2m+1}\alpha_i \quad \text{ and } \quad \d f_1 \wedge \d f_2 \wedge \d \alpha_i = 0\quad \text{ for }\ i=1,2.  \]
We distinguish two cases. For $m=1$ the induction hypothesis implies that
\[\d \alpha = \d g_1\wedge \d f_1 +\d g_2\wedge \d f_2,\]
and hence we can conclude that
\[ \alpha = \d g_0+ g_1 \d f_1 +g_2 \d f_2, \]
For $1<m $ the induction hypothesis yields 
\[\d \alpha = \d g_1\wedge \d f_1 +\d g_2\wedge \d f_2 + g\d f_1\wedge \d f_2,\]
for some $g \in \SR$.
Taking $\d$ of this expression we see that $g\mu\in \ker \delta_{\pi}$. Hence $g\in \R[[f_1,f_2]]$ and we can absorb $g$ into $g_1$ and $g_2$. Therefore we obtain
\[ \alpha = \d g_0 + g_1 \d f_1 +g_2 \d f_2\]
for some $g_0,g_1,g_2\in \SR$, which proves the Proposition for $n$ odd.

For the even case let $n=2m$. The base case ($m=0$) can be checked by hand, where we obtain that
\[ \alpha= \d g_0+\sum_{i=1}^2p_i\zeta_i +g_i\d f_i \]
with $g_0\in \SR$ homogeneous of degree $2$ and $p_i, g_i \in \R$ for $i=1,2$.  For $\alpha$ satisfying 
\[ m_t^*\alpha=t^{2m+2}\alpha \qquad \text{ and }\qquad \d f_1 \wedge \d f_2 \wedge \d \alpha = 0 \]
we get from Lemma \ref{lemma: induction degree 3} that
\[ \d \alpha = \d \left((c_1f_1^m+c_2f_1^{m-1}f_2) \zeta_2 \right) + \sum_{j=1}^2(\d g_j +\sum_{i=1}^2p_i^j\zeta_i) \wedge \d f_j + g\d f_1\wedge \d f_2.\]
Here the $p_i^j\in \R[[f_1,f_2]]$ are homogeneous of degree $2m-2$, $g_j\in \SR$ are homogeneous of degree $2m$ for $i,j\in \{1,2\}$ and $g\in \SR$ is homogeneous of degree $2m-2$. Hence we have (see \eqref{eq: modulo rhs})
\begin{align}\label{eq: mod ker}
    \d \alpha = \sum_{i,j=1}^2p_i^j\zeta_i \wedge \d f_j + g\d f_1\wedge \d f_2 \qquad \mod \dif V.
\end{align} 

\begin{lemma}
We may assume that
\[ \d \alpha = \sum_{i=1}^2 p_i \zeta_i \wedge \d f_1 + g \d f_1 \wedge \d f_2 \qquad \mod \dif V.\]
\end{lemma}
\begin{proof}
Note that we have the relations
\begin{align*}
    \dif f_1\wedge \zeta_1 -\dif f_2\wedge \zeta_2 =\d (f_1\zeta_1-f_2\zeta_2)\quad \text{ and }\quad \dif f_2\wedge \zeta_1 + \dif f_1\wedge \zeta_2 =\d (f_2\zeta_1+f_1\zeta_2).
\end{align*} 
Hence for any $p,q \in \R[[f_1,f_2]]$ we have:
\begin{align*}
    (p+f_2 \partial_y p) \d f_2 \wedge \zeta_2 -  f_1 \partial_yp \d f_2 \wedge \zeta_1 = (p+f_1 \partial_x p)\d f_1 \wedge \zeta_1 -  f_2\partial_x p \d f_1 \wedge \zeta_2  \qquad \mod \d V, \\
    (q+f_2\partial_y q) \d f_2 \wedge \zeta_1 +  f_1 \partial_y q \d f_2 \wedge \zeta_2 = -(q+ f_1 \partial_x q) \d f_1 \wedge \zeta_2 -  f_2 \partial_x q \d f_1 \wedge \zeta_1 \qquad \mod \d V.
\end{align*}
Adding those two equations, we want to determine the image of 
\[ \begin{pmatrix} p \\ q\end{pmatrix} \mapsto \begin{pmatrix} 
q - f_1 \partial_y p + f_2 \partial_y q \\ p + f_2 \partial_y p +  f_1 \partial_y q\\
\end{pmatrix}.\]
By looking at the maximal powers of $f_2$ in the image and inductively decreasing, we obtain that the map is injective, and hence surjective, thus proving the claim.
\end{proof}
Applying $\d$ to equation \eqref{eq: mod ker} yields
\[ 0 = \d(\sum_{i=1}^2p_i\zeta_i) \wedge \d f_1 + \d g\wedge \d f_1\wedge \d f_2. \]
To show that this implies \eqref{eq: ker degree 1} we first prove an auxiliary Lemma.
\begin{lemma}\label{lemma: dgt zero}
    For any $g\in \SR$ we have
    \[ \Lie_{T_1}g,\Lie_{T_2}g\in \R[[f_1,f_2]] \quad \Rightarrow \quad \Lie_{T_1}g=\Lie_{T_2}g=0.\]
\end{lemma}
\begin{proof}
    Note that we can focus on even homogeneous degrees only. Hence we can view $g$ as an element in ring generated by
    \begin{align*}
        f_1,\, f_2,\,  x_1^2+x_2^2+x_3^2+x_4^2,\, x_1^2-x_2^2-x_3^2+x_4^2,\, x_1^2+x_2^2-x_3^2-x_4^2, \qquad \qquad \\
        2(x_1x_2-x_3x_4),\, 2(x_1x_3+x_2x_4),\, 2(x_1x_3-x_2x_4),\, 2(x_1x_4+x_2x_3),\, 2(x_1x_4-x_2x_3). 
    \end{align*}
    By a direct computation we can check that $\Lie_{T_i}$ maps generators to generators. More generally, for a homogeneous monomial of degree $n$ in the above variables we note that either the total degree $n_f$ in the variables of $f_1$ and $f_2$ is preserved or the monomial is maped to zero. Therefore it is enough to look at polynomomials in $f_1$ and $f_2$, but on such $\Lie_{T_i}$ acts trivial.
\end{proof}
We use the Lemma to show the following.
\begin{corollary}\label{lemma: ps are zero with d}
    Fix $j=1,2$. For $p_1,p_2\in \R[[f_1,f_2]]$ and $g\in \SR$ the condition
    \begin{align*}
        \dif g\wedge\dif f_1\wedge\dif f_2  + \sum_{i=1}^2 \dif (p_i  \zeta_i)\wedge \dif f_j = 0 
    \end{align*}
    implies $p_1,p_2=0$ and $g\in \R[[f_1,f_2]]$.
\end{corollary}
\begin{proof}
    Let $j=1$. Note that we have the relations 
    \begin{align}\label{eq: rel T, zeta contract}
        4\iota_{T_i}\zeta_1=f_i, \quad 4\iota_{T_1}\zeta_2=f_2 \quad  4\iota_{T_2}\zeta_2=-f_1 \qquad \Lie_{T_i}\zeta_j=0.
    \end{align}
    Hence contracting with $T_1$ and $T_2$, respectively together with \eqref{eq: relation E, T and f} yields the equations
    \begin{align*}
        0=f_1\partial_y p_1+f_2\partial_yp_2+p_2 +4\dif g(T_1) \quad \text{ and }\quad 0=f_2\partial_y p_1 +p_1- f_1\partial_y p_2 +4\dif g(T_2) 
    \end{align*}
    and therefore $\dif g(T_1),\dif g(T_2)\in \R[[f_1,f_2]]$. Hence Lemma \ref{lemma: dgt zero} implies $\Lie_{T_i}(g)=0$ and the above equations imply $p_1=p_2=0$. The statement for $g$ then follows from \eqref{eq: ker deg 4}. The proof for $j=2$ follows along the same lines.
\end{proof}
\begin{remark}\label{remark: adapted lemma}
    Note that the conclusion of Corollary \ref{lemma: ps are zero with d} also holds under the assumption
    \begin{align*}
        \dif g\wedge\dif f_1\wedge\dif f_2  + \sum_{i=1}^2 p_i  \d\zeta_i\wedge \dif f_j = 0 
    \end{align*}
    The proof follows along the same lines.
\end{remark}

\subsection{Proof of Degree $3$}
The argument consists of two steps which are proven below. Suppose that $\gamma \in \Omega^3_f$ satisfies $\delta_\pi \gamma = 0$ then we show the following:
\begin{enumerate}
    \item[\textbf{Step 1}:] Replacing $\gamma$ by $\gamma - q_1 \epsilon_1 - q_2 \epsilon_2$ for some $q_1,q_2 \in \R[[f_1,f_2]]$ we may assume that
    \begin{equation}\label{eq:DiverenceCondition} \star^{-1} (\d \gamma) = \iota_{\star^{-1}\gamma}\d f_1 = \iota_{\star^{-1}\gamma}\d f_2 = 0.
    \end{equation}
    \item[\textbf{Step 2}:] If $\gamma$ satisfies \eqref{eq:DiverenceCondition} then there exists $\alpha \in \ker \delta_\pi \subset \Omega^1_f$ such that
    \[ \gamma = \d f_1 \wedge \d \alpha.\]
\end{enumerate}
Using \eqref{eq: ker degree 1} the last equation can be rewritten as
\[\gamma = \d f_1 \wedge \d\big(f_2\d g + \sum_{i=1}^2 p_i \zeta_i\big),\]
for $g \in \SR$ and $p_1,p_2 \in \R[[f_1,f_2]]$, which finishes the proof.

\textbf{\underline{Proof of Step 1:}}
Let $\gamma \in \Omega^3_f$ be such that $m_t^*\gamma= t^{m+3}\gamma$ and $\delta_\pi \gamma = 0$, i.e.
\begin{align*}
    0 = \delta_{\pi} \gamma =\iota_{\pi} (\d \gamma) - \d\iota_{\star^{-1}\gamma}(\d f_1 \wedge\d f_2).
\end{align*}
Wedging the above equation with $\dif f_1$ and $\dif f_2$ respectively, using Proposition \ref{proposition: Poisson differential} and \eqref{eq: ker deg 4} implies
\begin{align*}
        \iota_{\star^{-1}\gamma}(\dif f_i) =p_i \in \R[[f_1,f_2]],
    \end{align*}
    satisfying the equation
    \begin{align}\label{eq: transverse degree 1}
        \star^{-1}(\d \gamma) \d f_1\wedge \d f_2 =\big((\partial_x p_1) (f_1,f_2)+(\partial_y p_2) (f_1,f_2)\big) \d f_1\wedge \d f_2
    \end{align}
    and hence in particular that $\star^{-1}(\d \gamma)\in \R[[f_1,f_2]]$. Using the relations \eqref{eq: relation E, T and f}
    \begin{align*}
        \delta(E_1)=2,\quad \delta(E_2)=0, \quad  
    \end{align*}
    it is easy to see that $\epsilon_i=\star E_i\in \ker \delta_\pi$. Given $q_1,q_2\in \R[[f_1,f_2]]$, the identities above also imply:
    \begin{align*}
        \iota_{q_1E_1+q_2E_2}\dif f_1=&\ (xq_1 +yq_2)(f_1,f_2), \\ 
        \iota_{q_1E_1+q_2E_2}\dif f_2=&\ (yq_1 -xq_2)(f_1,f_2). 
    \end{align*}
    Hence, replacing $\gamma $ by $\gamma -q_1 \epsilon_1 - q_2\epsilon_2$, for suitably chosen $q_1$ and $q_2$, we may assume that
    \begin{align*}
        p_1(f_1,f_2)= p_1(f_1) \quad \text{ and }\quad p_2(f_1,f_2)= p_2(f_1).
    \end{align*}
    We want to show that $p_1 = p_2 = 0$. Recall that be homogeneity of $\gamma$ and since $f_1$ and $f_2$ are homogeneous of degree $2$, we have $m= 2n-1$ where $n\geq 1$ is the homogeneous degree of $p_i$ as a polynomial in one variable. As such we can write
    \[ p_1 = c_1 x^n\quad \text{and}\quad p_2 = c_2 x^n,\]
    for some $c_1,c_2 \in \R$. We also write $\gamma = \sum_{i=1}^4 \gamma_i\star( \partial_{x_i})$. We denote by $a_i$ and $b_i$ respectively the coefficient of the $x_1^{2n-1-i}x_2^i$-terms for $\gamma_1$ and $\gamma_2$, $0\leq i\leq 2n-1$. With this notation, explicitly computing the $x_1^{2n-i}x_2^i$-terms of Equation \eqref{eq: transverse degree 1} yields the following conditions:
    \begin{align*}
        2a_0 = c_1 \quad &\ \text{ and } \quad 2b_0=c_2,\\
        (-1)^j\binom{n}{n-j}c_1= 2(a_{2j}+  b_{2j-1}) \quad &\ \text{ and } \quad (-1)^j\binom{n}{n-j}c_2= 2(a_{2j-1}+b_{2j})\quad &\ \text{ for }\ 1\le j\le n,\\
        a_{2j+1}= -b_{2j} \quad &\ \text{ and } \quad  a_{2j}=-b_{2j+1}\quad \text{ for }\ 0\le j\le n-1,\\
        2b_{2n-1}=(-1)^{n-1}c_1 \quad &\ \text{ and }\quad 2a_{2n-1}=(-1)^nc_2.
    \end{align*}
    Using the equations in lines 1-3 we obtain 
    \begin{align*}
        a_{2j+1}= - b_{2j} =c_2\frac{(-1)^j}{2}\sum_{0\le l\le j}\binom{n}{l} \quad \text{ and }\quad  a_{2j}=-b_{2j+1}= c_2\frac{(-1)^j}{2}\sum_{0\le l\le j}\binom{n}{l}
    \end{align*}
    for $1\le j\le n-1$, and hence in particular that
    \begin{align*}
        a_{2n-1}=c_2\frac{(-1)^{n-1}}{2}(2^n-1) \quad \text{ and }\quad b_{2n-1}=c_1\frac{(-1)^{n}}{2}(2^n-1)
    \end{align*}
    which yields a contradiction for all $1\le n$ unless $c_1=c_2=0$. This completes the proof of Step 1.

\textbf{\underline{Proof of Step 2:}} Assume $\gamma$ satisfies \eqref{eq:DiverenceCondition}. Observe that for any $\alpha \in \Omega^1_f$ we have
\begin{align}\label{eq: canonical iso}
    \star\circ \sharp \ (\iota_{\star^{-1}\gamma}(\alpha)) = \alpha \wedge \gamma
\end{align} 
where $\sharp: \Omega^0_f\to \mathfrak{X}^0_f$ is the canonical identification. Hence Corollary \ref{lemma: exactness wedge f} implies that
\begin{align*}
    \gamma= \dif f_1 \wedge \beta
\end{align*}
for some $\beta \in \mathcal{D}^2(\d f_1,\d f_2)$ with $m_t^*\beta = t^{m+1}\beta$. Thus, it remains to show that $\beta=\d\alpha$ with $\alpha \in \ker \delta_\pi$. From Proposition \ref{proposition: iso division group} we know that
\begin{align*}
    \beta= I(c,p,q)+\sum_{i=1}^2\dif f_i \wedge \alpha_i 
\end{align*}
for some $c\in \R, p\in \R[[x_2]], q\in\R[[x_2,x_4]]^2$ and $\alpha_i\in \Omega^1(\SR)$. Moreover, since we are interested in $\gamma$ we may assume that $\alpha_1 = 0$, and rename $\alpha_2$ to $\alpha$. Then Equation \eqref{eq:DiverenceCondition} implies:
\begin{align}\label{eq: tangent degree 13 general}
    \dif f_1 \wedge \dif I(c,p,q)=\dif f_1\wedge \dif f_2 \wedge \dif \alpha.
\end{align}
Hence it suffices to show that
\[ \beta= \sum_{i=1}^2\dif f_i \wedge \alpha_i \mod \d V.\]
That is, we may assume $I(c,p,q) = 0$. If $\beta$ has constant coefficients,  i.e.\ $m=1$, then this follows immediately from the fact that $\dif \zeta_i =\beta_i$. Hence let's assume $\beta$ is homogeneous of degree $m>1$. To see the statement, we recall that the coefficients of $\d f_1 \wedge \d f_2$ are (by definition) in the ideal $\SJ$ defined in \eqref{eq: ideal df12}, so that \eqref{eq: tangent degree 13 general} is zero $\operatorname{mod} \SJ$. This means the left hand side of \eqref{eq: tangent degree 13 general} satisfies the hypothesis of Lemma \ref{ex: general} taking $g \cdot \mu = 0$, which in turn implies that $x_2 p + x_4q_1, q_2 \in R[[x_2^2 + x_4^2]]$. Hence $m=2n+1$ for $1\le n$, otherwise we are done. It then follows from Corollary \ref{lemma: relation second cohomology} that $[(x_1p + x_3q_1)\beta_2] = [c_1f_1^{n-1}f_2\beta_2]$ and $[q_2\beta_2] = [c_2f_1^n \beta_2]$ so that 
\[ I(0,p,q) = [(c_1f_1^{n-1}f_2 + c_2f_1^n)\beta_2].\]
Finally, since $\d \zeta_i = \beta_i$, and $\d(c_1 f_1^{n-1}f_2 + c_2f_1^n) = 0 \ \operatorname{mod} \SJ$ we obtain
\[ I(0,p,q)= [\d((c_1f_1^{n-1}f_2 + c_2 f_1^n)\zeta_2)],\]
as desired. Hence the proof is complete.

\subsection{Proof of Degree $2$}\label{section: kernel degree 2 proof}

Let $\beta \in \Omega^2_f$ such that
\begin{align*}
    \delta_{\pi}(\beta) =&\iota_{\star^{-1}\dif\beta}(\dif f_1\wedge \dif f_2) -\dif \iota_{\pi}(\beta)=0.
\end{align*}
 Applying $\d $ to the equation implies that $\d \beta \in \ker \delta_{\pi}$. By a direct computation we can check that
 \[ q\zeta_1\wedge \zeta_2 \in \ker \delta_\pi \]
 for any $q\in \R[[f_1,f_2]]$. Moreover, the relations
 \begin{align}\label{eq: relation zeta f E}
     \d f_1\wedge \zeta_1 \wedge \zeta_2 = f_1 \epsilon_2 -f_2\epsilon_1, \quad \d f_2\wedge \zeta_1 \wedge \zeta_2 = f_2 \epsilon_2 +f_1\epsilon_1 \quad \text{ and }\quad \dif (\zeta_1\wedge \zeta_2)  = 2 \epsilon_2
 \end{align} 
imply that we have
\begin{align}\label{eq: d zeta wedge}
    \d (q\zeta_1\wedge \zeta_2)= (2q +f_1\partial_x q +f_2\partial_y q)\epsilon_2 +( f_1\partial_yq-f_2\partial_xq)\epsilon_1.
\end{align} 
Therefore, by replacing $\beta$ by $\beta -q\zeta_1\wedge \zeta_2 $ and using \eqref{eq: ker degree 3} we can assume that 
\[ \d \beta =  \dif f_1\wedge\dif f_2\wedge \dif g + q\epsilon_1 +\dif f_1\wedge \dif  \big(p_i\zeta_i\big)  \]
for some $q,p_i \in \R[[f_1,f_2]]$ and $g\in\SR$. Taking the differential $\d$ of this equation yields
\[ 0 = \d q \wedge \epsilon_1 + 2 q\mu =(f_1\partial_x q+f_2\partial_yq +2q)\mu \]
and hence $q=0$ and we can conclude that 
\[ \beta=g\dif f_1\wedge\dif f_2 +\dif f_1\wedge p_i\zeta_i +\d \alpha  \]
for $p_i \in \R[[f_1,f_2]]$, $g\in\SR$ and $\alpha\in \Omega^1_f$. In other words, the map
\begin{align*}
    \d: \{ \beta=g\dif f_1\wedge\dif f_2 +\dif f_1\wedge p_i\zeta_i +\d \alpha\}\cap \ker \delta_\pi \to \{ \gamma \in \ker \delta_\pi | \ \gamma \text{ satisfies } \eqref{eq: transverse degree 1} \}
\end{align*}
is surjective. Therefore, it is enough to study its kernel, i.e.\ we may assume $\beta= \d \alpha$ such that
\[ \delta_{\pi}(\beta) =\iota_{\star^{-1}\dif\beta}(\dif f_1\wedge \dif f_2) -\dif \iota_{\pi}(\beta)=-\dif \iota_{\pi}(\d \alpha)=0.\]
By Proposition \ref{proposition: Poisson differential} this implies that $\alpha \in \ker \delta_\pi$ and hence \eqref{eq: ker degree 1} implies the statement.

\section{Poincare series}\label{sec: poincare series}
In this section we use the Hilbert-Poincare series to determine the dimension of the Poisson homology groups. We achieve this by constructing short exact sequences for the kernels of $\delta_\pi$ as described in Proposition \ref{proposition: kernel differential}. We obtain the following result:

\begin{proposition}\label{proposition: Poincare series}
For the Poisson homology spaces we have the following Hilbert-Poincare series:
\begin{align*}
    HP_{H_0}(t)=&\ \frac{4t^2+4t+1}{(1-t^2)^2},\\
    HP_{H_1}(t)=&\ 4t\frac{t^3+t^2+2t+1}{(1-t^2)^2},\\
    HP_{H_2}(t)=&\ 2t^2\frac{4t^2+2t+1}{(1-t^2)^2}.
\end{align*}
\end{proposition}
In degree $1$ we have the following statement.
\begin{lemma}\label{lemma: ses degree 3}
The following sequence is exact:
\begin{align*}
    0 \to \R[[f_1,f_2]] \xrightarrow{\theta_1} \SR(-2)^2\oplus \SR \oplus \R[[f_1,f_2]](-2)^2 \xrightarrow{\beta} \ker \delta_{\pi,1} \to 0,
\end{align*}
where the maps are given by
\begin{align*}
    \theta_1:  u \mapsto \left((-\partial_{f_1} u, - \partial_{f_2} u), u,0\right)\quad \text{ and }\quad
    \sigma_1: ((g_1,g_2),g_0,(p_1,p_2)) \mapsto \d g_0 + \sum_{i} p_i(f_1,f_2) \zeta_i + g_i \d f_i.
\end{align*}
\end{lemma}
\begin{proof}
Injectivity of $\theta_1$ follows from the definition, while surjectivity of $\sigma_1$ follows from \eqref{eq: ker degree 1}. Let
\begin{align*}
    \alpha:= \sigma_1(g_1,g_2,g_0,p_1,p_2) = 0.
\end{align*}
Wedging with $\dif f_1$, applying $\dif$ to both sides and using Corollary \ref{lemma: ps are zero with d} implies that $p_1=p_2=0$ and $g_2\in \R[[f_1,f_2]]$. Wedging with $\dif f_1\wedge \dif f_2 $ implies that $g_0\in \R[[f_1,f_2]]$ by Proposition \ref{proposition: Poisson differential} and \eqref{eq: ker deg 4}. Hence the statement follows by contraction with $E_1$ and $E_2$ by \eqref{eq: relation E, T and f}.
\end{proof}
In degree $2$ we obtain the following short exact sequences for the kernel.
\begin{lemma}\label{lemma: ses degree 2}
The following sequence is exact:
\begin{align*}
    0 \to \R[[f_1,f_2]](-2)^2 \xrightarrow{\theta_2} \SR(-2)^2\oplus \SR(-4) \oplus \R[[f_1,f_2]](-4)^3\oplus \R[[f_1,f_2]](-2)^2 \xrightarrow{\sigma_2 } \ker \delta_{\pi,2} \to 0,
\end{align*}
where the maps are given by
\begin{align*}
    &\ \theta_2: (u_1,u_2) \mapsto \left((u_1, u_2),\partial_{f_2} u_1-\partial_{f_1} u_2,0,0\right) \quad \text{ and }\\
    \sigma_2: &\ ((g_1,g_2),g,(p,p_1,p_2),(q_1,q_2)) \mapsto p\zeta_1\wedge\zeta_2+g\dif f_1\wedge\dif f_2 +\sum_{i=1}^2 \dif f_1\wedge p_i \zeta_i +\dif (q_{i}\zeta_i) +\dif g_i\wedge \dif f_i.
\end{align*}
\end{lemma}
\begin{proof}
Note that $\theta_2$ is injective, $\sigma_2$ is surjective by the description of $\ker \delta_{\pi,2}$ in \eqref{eq: ker degree 2} and we have 
\begin{align*}
    \im \theta_2 \subset \ker \sigma_2 .
\end{align*}
Let us determine the kernel of $\sigma_2$. To do so, we set
\begin{align*}
    \beta:= \sigma_2((g_1,g_2),g,(p,p_1,p_2),(q_1,q_2)).
\end{align*}
We first note that due to \eqref{eq: d zeta wedge}, \eqref{eq: relation E, T and f} and \eqref{eq: contraction vs wedge} we have
\begin{align*}
    \dif \beta \wedge (f_2\dif f_1 -f_1\dif f_2)=2(f_1^2+f_2^2)(2p +f_1\partial_x p +f_2\partial_y p)\mu,
\end{align*}
implying that $p=0$ if $\beta=0$. Wedging with $\dif f_1$ and applying Corollary \ref{lemma: ps are zero with d} implies that $q_1=q_2=0$ and $g_2\in \R[[f_1,f_2]]$. Applying $\dif$ and using again Corollary \ref{lemma: ps are zero with d} implies the statement.
\end{proof}
In degree $3$ we obtain the following short exact sequence for the kernel of the differential.
\begin{lemma}\label{lemma: ses degree 1}
The following sequence is exact:
\begin{align*}
    0 \to \R[[f_1,f_2]](-4) \xrightarrow{\theta_3} \SR(-4) \oplus \R[[ f_1,f_2]](-4)^4  \xrightarrow{\sigma_3} \ker \delta_{\pi,3} \to 0,
\end{align*}
where the maps are given by:
\begin{align*}
    \theta_3: u \mapsto (u,0),\quad \text{ and }\quad  \sigma_3 :(g,(q_1,q_2,p_1,p_2)) \mapsto  \dif f_1\wedge\dif f_2\wedge \dif g + \sum_{i=1}^2 q_i \epsilon_i+ \dif f_1\wedge \dif (p_i  \zeta_i).
\end{align*}
\end{lemma}
\begin{proof}
The map $\theta_3$ is injective by definition and surjectivity of $\sigma_3 $ follows from \eqref{eq: ker degree 3}. To show that
\begin{align*}
    \im \theta_3=\ker \sigma_3
\end{align*}
we define 
\begin{align*}
    \gamma := \sigma_3 (g,(q_1,q_2,p_1,p_2)).
\end{align*}
Note that by \eqref{eq: relation E, T and f} and \eqref{eq: contraction vs wedge} we have the relations
\begin{align*}
    (f_1\dif f_1 +f_2\dif f_2)\wedge \gamma =2(f_1^2+f_2^2)q_1\mu, \quad \text{ and } \quad (f_2\dif f_1 -f_1\dif f_2)\wedge \gamma =2(f_1^2+f_2^2)q_2\mu.
\end{align*}
Hence $\gamma =0$ implies that $q_1=q_2=0$. Therefore Corollary \ref{lemma: ps are zero with d} implies the statement.
\end{proof}

Now we are ready to proof Proposition \ref{proposition: Poincare series}.
\begin{proof}[Proof of Proposition \ref{proposition: Poincare series}]
We use the exact sequence of degree preserving maps
\begin{align}\label{eq: trivial es}
    0\to \ker \delta_{\pi,i+1} \hookrightarrow\Omega^{i+1}_f\xrightarrow{\delta_{\pi,i+1}}\ker \delta_{\pi,i}\to H_i \to 0.
\end{align}
Let us first compute the Hilbert-Poincare series of the kernels of $\delta_{\pi,i}$. Therefore we note that
\begin{align*}
    HP_{\SR(-n)}=\frac{t^n}{(1-t)^4} \quad \text{ and }\quad HP_{\R[[f_1,f_2]](-n)}=\frac{t^n}{(1-t^2)^2}
\end{align*}
for all $n\in \N_0$. For any two graded vector spaces $U$ and $V$ we have that
\begin{align*}
    HP_{U\oplus V}=HP_U+HP_V,
\end{align*}
and for any exact sequence of graded vector spaces
\begin{align*}
    0\to V_0\to \dots \to V_k \to 0
\end{align*}
with degree preserving maps in between, that
\begin{align*}
    HP_{V_k}=\sum_{i=0}^{k-1}(-1)^{k-1-i}HP_{V_i}.
\end{align*}
From Lemma \ref{lemma: ses degree 1} we obtain
\begin{align*}
    HP_{\ker \delta_{\pi,3}}(t)=\frac{t^4}{(1-t^2)^2}+\frac{t^4}{(1-t)^4}.
\end{align*}
Similarly, Lemma \ref{lemma: ses degree 2} and Lemma \ref{lemma: ses degree 3} imply
\begin{align*}
    HP_{\ker \delta_{\pi,2}}(t)=\frac{3t^4}{(1-t^2)^2}+\frac{t^4+2t^2}{(1-t)^4}\quad \text{ and }\quad HP_{\ker \delta_{\pi,1}}(t)=\frac{2t^2-1}{(1-t^2)^2}+\frac{2t^2+1}{(1-t)^4}.
\end{align*}
Additionally, note that by Example \ref{ex: HP formal forms} we have
\begin{align*}
    HP_{\ker \delta_{\pi,0}}(t)=HP_{\Omega^0_f}(t)=\frac{1}{(1-t)^4}, &\ \qquad HP_{\Omega^1_f}(t)=\frac{4t}{(1-t)^4}, \\
    HP_{\Omega^2_f}(t)=\frac{6t^2}{(1-t)^4},&\ \qquad HP_{\Omega^3_f}(t)=\frac{4t^3}{(1-t)^4}.
\end{align*}
Hence \eqref{eq: trivial es} for $i=2$ implies that
\begin{align*}
    HP_{H_2}(t)=&\ HP_{\ker \delta_{\pi,2}}(t)+HP_{\ker \delta_{\pi,3}}(t)-HP_{\Omega^3_f}(t)\\
    =&\ \frac{3t^4}{(1-t^2)^2}+\frac{t^4+2t^2}{(1-t)^4}+\frac{t^4}{(1-t^2)^2}+\frac{t^4}{(1-t)^4}-\frac{4t^3}{(1-t)^4}\\
    =&\ \frac{4t^4}{(1-t^2)^2}+2t^2\frac{(1-t)^2}{(1-t)^4}\ =\ 2t^2\frac{4t^2+4t+1}{(1-t^2)^2}.
\end{align*}
For $i=1$ we get
\begin{align*}
    HP_{H_1}(t)=&\ HP_{\ker \delta_{\pi,1}}(t)+HP_{\ker \delta_{\pi,2}}(t)-HP_{\Omega^2_f}(t)\\
    =&\ \frac{2t^2-1}{(1-t^2)^2}+\frac{2t^2+1}{(1-t)^4}+\frac{3t^4}{(1-t^2)^2}+\frac{t^4+2t^2}{(1-t)^4}-\frac{6t^2}{(1-t)^4}\\
    =&\ \frac{3t^4+2t^2-1}{(1-t^2)^2}+\frac{(1-t)^2(1+t)^2}{(1-t)^4}\ =\ 4t\frac{t^3+t^2+2t+1}{(1-t^2)^2}.
\end{align*}
Finally, for $i=0$ we compute
\begin{align*}
    HP_{H_0}(t)=&\ HP_{\ker \delta_{\pi,0}}(t)+HP_{\ker \delta_{\pi,1}}(t)-HP_{\Omega^1_f}(t)\\
    =&\ \frac{1}{(1-t)^4}+\frac{2t^2-1}{(1-t^2)^2}+\frac{2t^2+1}{(1-t)^4}-\frac{4t}{(1-t)^4}\\
    =&\ \frac{2t^2-1}{(1-t^2)^2}+2\frac{(1-t)^2}{(1-t)^4}\ =\ \frac{4t^2+4t+1}{(1-t^2)^2}.
\end{align*}
\end{proof}

\section{Proof of the main theorem}\label{sec:proof main thm}

In this section we prove Theorem \ref{main theorem}. In each degree the strategy of the proof is the same: by Proposition \ref{proposition: Poincare series} the given set of generators has the right dimension. Hence it is enough to show that none of the elements generated by the given module are in the image of the Poisson differential.

\subsection{Degree 0} For the proof in degree $0$ we need the following auxiliary statement.
\begin{lemma}\label{lemma: degree 4 proof}
Let $g\in \R[[f_1,f_2]]$, $(c,p,(q_1,q_2))\in \R\oplus \R[[x_2]] \oplus \R[[x_2,x_4]]^2$ and $\alpha \in \Omega^1_f$ such that:
\begin{align*}
    g\cdot \mu =&\, \dif f_1 \wedge \dif I(c,p,(q_1,q_2)) +\dif f_1\wedge \dif f_2 \wedge \dif \alpha .
\end{align*}
Then $g=0$ and the same statement holds if we exchange the roles of $\dif f_1$ and $\dif f_2$.    
\end{lemma}
\begin{proof}
We prove the statement for $g$ being homogeneous of degree $2n$ by induction on $n$.

\noindent\underline{$n=0 ,1$\,:} If $g\in \R$ then the statement holds by a degree count of the coefficients. If $g$ is homogeneous of degree $1$ then the statement follows from Lemma \ref{ex: general}.

\noindent\underline{Induction step:} Assume that $g\in \R[[f_1,f_2]]$ is homogeneous of degree $2n$. Then Lemma \ref{ex: general} implies 
\[ x_2p+x_4q_1, q_2\in \R\cdot (x_2^2+x_4^2)^{n-1}, \qquad \text{ and } \qquad g\in (f_1^2+f_2^2)\R[[f_1,f_2]].\]
Hence, by Corollary \ref{lemma: relation second cohomology} we have:
\begin{align*}
    I(c,p,(q_1,q_2))= (c_1f_1^{n-1}+c_2f_1^{n-2}f_2 )\beta_2 +\sum_{i=1}^{2}\dif f_i \wedge \alpha_i,
\end{align*}
for some $c_i \in \R$ and $\alpha_i \in \Omega^1_f$, with $i=1,2$. Therefore, writing $g$ as $(f_1^2+f_2^2)\wtd{g}$ and replacing $\alpha$ by $\alpha-\alpha_2$, again denoted by $\alpha$, the original equation becomes
\begin{align*}
    (f_1^2+f_2^2)\wtd{g}\cdot \mu =&\, \dif f_1\wedge \dif f_2 \wedge \dif \alpha.
\end{align*}
Using \eqref{eq: relation E, T and f}, \eqref{eq: contraction vs wedge} and \eqref{eq: relation zeta f E} we note that
\begin{align*}
    2g\dif f_1\wedge\dif f_2\wedge \zeta_1\wedge\zeta_2 = (f_1^2+f_2^2)g\cdot \mu .
\end{align*}
Hence combining the two equations above, we obtain by Proposition \ref{proposition: iso division group} that
\begin{align*}
    2\wtd{g}\zeta_1\wedge \zeta_2 -\dif \alpha = I(0,\wtd{p},(\wtd{q}_1,\wtd{q}_2)) +\sum_{i=1}^2\alpha_i\wedge \dif f_i
\end{align*}
for some $(\wtd{p},(\wtd{q}_1,\wtd{q}_2))\in \R[[x_2]]\times \R[[x_2,x_4]]^2$ and $\alpha_i\in \Omega^1_f$. Note that the homogeneous degrees of $\wtd{p},\wtd{q}_1,\wtd{q}_2$ and $\alpha_i$ are two smaller than those of $p,q_1,q_2$ and $\alpha$, respectively. Now taking the de-Rham differential and wedging with $\dif f_1$ implies by \eqref{eq: relation E, T and f}, \eqref{eq: contraction vs wedge} and \eqref{eq: relation zeta f E} that
\begin{align*}
    2((f_1^2+f_2)^2\partial_{y}\wtd{g} +f_2\wtd{g})\cdot \mu = \dif f_1\wedge \dif I(0,\wtd{p},(\wtd{q}_1,\wtd{q}_2))+\dif f_1\wedge \dif f_2\wedge \dif \alpha_2.
\end{align*}
Therefore, the induction hypothesis implies that
\begin{align*}
    (f_1^2+f_2)^2\partial_{y}\wtd{g} +f_2\wtd{g}=0,
\end{align*}
and hence $\wtd{g}=0$ concluding the induction. The proof for $\dif f_2$ follows along the same lines.
\end{proof}

Now we are ready to prove Theorem \ref{main theorem} for degree $0$. We distinguish between even and odd homogeneous degrees. 

\noindent\underline{The odd degree:} The zeroth Poisson homology group is realized by the free $\R[[x_2^2,x_4^2]]$-module generated by the linear monomials. This follows from the description of a boundary by Proposition \ref{proposition: Poisson differential} and Lemma \ref{ex: general 1}.

\noindent\underline{The even degree:}
By a similar argument as in the odd degree we obtain that the coefficients in even degree are contained in the ideal $\SI$ if and only if they are of the form
\begin{align*}
    (f_1^2+f_2^2)p
\end{align*}
where $p\in \R[[f_1,f_2]]$. Hence we only need to check if there exists an $\alpha\in \Omega^1(\SR)$ such that
\begin{align*}
    (f_1^2+f_2^2)p\mu=\dif f_1\wedge \dif f_2\wedge \dif \alpha
\end{align*}
We show that such an $\alpha$ can only exist if $p=0$. This follows from Lemma \ref{lemma: degree 4 proof}.

\subsection{Degree 1} We first prove another auxiliary lemma.
\begin{lemma}\label{lemma: vector field normal form}
Let $g\in \R[[f_1,f_2]]$ and $a_i \in \R[[x_2^2,x_4]]$ and $\gamma \in\Omega^3_f$ such that
\begin{equation*}
    0= \left(g-\star^{-1}(\d \gamma)\right)\dif f_1\wedge \dif f_2 +\dif \left( \left(\sum_{i=1}^4 a_ix_i-\iota_{\star^{-1} \gamma}(\d f_2)\right)\dif f_1+\iota_{\star^{-1}\gamma}(\d f_1)\dif f_2\right). 
\end{equation*}
Then $g=a_1=a_2=a_3=a_4=0$ and $\gamma \in \ker \delta_{\pi}$. 
\end{lemma}

\begin{proof}
Note that by wedging with $\dif f_1$, Proposition \ref{proposition: Poisson differential} and \eqref{eq: ker deg 4} imply
\[ \iota_{\star^{-1}\gamma}(\d f_1)\in \R[[f_1,f_2]]. \]
Using \eqref{eq: canonical iso}, \eqref{eq: relation E, T and f} and Corollary \ref{lemma: exactness wedge f} we obtain that 
\begin{align*}
     \gamma =r_1\epsilon_1+ r_2\epsilon_2+\dif f_1\wedge \beta
\end{align*}
for unique $r_1\in \R[[f_1,f_2]]$, $r_2\in \R[[f_2]]$ and some $\beta\in \Omega^2_f$. Wedging with $\dif f_2$ and using \eqref{eq: relation E, T and f}, \eqref{eq: canonical iso}, Proposition \ref{proposition: Poisson differential} and \eqref{eq: ker deg 4} implies
\begin{align*}
    \left(\sum_{i=1}^4 a_ix_i\right)\mu+\beta\wedge \dif f_1 \wedge \dif f_2 \in \R[[f_1,f_2]]\mu.
\end{align*}
Hence Lemma \ref{ex: general 1} yields
\begin{align*}
    a_1=a_2=a_3=a_4=0 \qquad \text{ and }\qquad \beta\wedge \dif f_1 \wedge \dif f_2 \in (f_1^2+f_2^2)\R[[f_1,f_2]]\mu.
\end{align*}
This implies that 
\[\beta = r \zeta_1\wedge \zeta_2+ \wtd{\beta} + \d f_1 \wedge \alpha_1 + \d f_2 \wedge \alpha_2,\]
for some $r\in \R[[f_1,f_2]]$, $\alpha_1,\alpha_2 \in \Omega^1_f$ and $\wtd{\beta}$ representing a non-trivial class in $\mathcal{D}^2(\d f_1,\d f_2)$. Note that 
\begin{align*}
    r_1\epsilon_1+ r_2\epsilon_2+r\d f_1\wedge  \zeta_1\wedge \zeta_2 \in \ker \delta_\pi .
\end{align*}
Hence it is enough to show 
\begin{align*}
    \gamma:=\dif f_1\wedge (\wtd{\beta} + \d f_2 \wedge \alpha_2) \in \ker \delta_\pi.
\end{align*}
In this case, the original equation is equivalent to 
\begin{align*}
    g\cdot \mu = \dif f_1 \wedge \dif \wtd{\beta} +\dif f_1\wedge \dif f_2 \wedge \dif \alpha_2.
\end{align*}
Using Lemma \ref{lemma: degree 4 proof} we obtain $g=0$, which concludes the proof.
\end{proof}
To prove the statement in degree one, it is, by a dimension count and Proposition \ref{proposition: Poincare series}, enough to show that for any non-trivial choice of coefficients $a_j,b_j\in \R[[x_2^2,x_4]]$ and $q_i,p_i\in \R[[f_1,f_2]]$, the $1$-cycle
\begin{align*}
    \alpha:= \dif (\sum_{j=1}^4 a_{j}x_j) +\sum_{j=1}^4 b_{j}x_j\dif f_1 +\sum_{i=1}^2q_i\dif f_i + p_i \zeta_i
\end{align*}
is not a boundary. That is, for all $\beta\in \Omega^2_f$ we have 
\[ \wtd{\alpha}:= \alpha- \delta_{\pi}(\beta) = 0 \quad \Rightarrow \quad a_j = b_j = q_i =p_i= 0.\]
If we take the differential of $\wtd{\alpha}=0$ we obtain:
\begin{align*}
    0=\dif \wtd{\alpha} = \d \left(\sum_{i=1}^2q_i\dif f_i + p_i \zeta_i\right)  +\dif \left(\sum_{j=1}^4 b_{j}x_j+ \iota_{\star^{-1}\d \beta}(\d f_2)\right)\wedge \dif f_1-\d \iota_{\star^{-1}\d \beta}(\d f_1)\wedge \dif f_2.
\end{align*}
Wedging with $\d f_1$, Corollary \ref{lemma: ps are zero with d} implies that $p_i=0$. Applying Lemma \ref{lemma: vector field normal form} for $\gamma=-\d \beta$ implies that $b_{j}=0$ for $j=1,2,3,4$, $\partial_x q_2= \partial_y q_1$ and $\d \beta\in \ker \delta_\pi$. The arguments from Section \ref{section: kernel degree 2 proof} give
\begin{align*}
    \beta = r\zeta_1\wedge \zeta_2 +g \dif f_1\wedge\dif f_2 +\dif f_1\wedge p_i\zeta_i +\d \overline{\alpha},
\end{align*}
for $r,p_i\in \R[[f_1,f_2]]$, $g\in \SR$ and $\overline{\alpha}\in \Omega^1_f$. Note all terms except $\d \overline{\alpha}$ are in $\ker \delta_\pi$. Therefore, we have
\begin{align*}
    0=\wtd{\alpha} = q_1\d f_1 +q_2\d f_2 +\d \left(\iota_{\star^{-1}\d \overline{\alpha}}(\d f_1\wedge \d f_2) +\sum_{j=1}^4 a_{j}x_j\right).
\end{align*}
Wedging the expression with $\d f_1\wedge \d f_2$ yields
\[ \d \left(\iota_{\star^{-1}\d \overline{\alpha}}(\d f_1\wedge \d f_2) +\sum_{j=1}^4 a_{j}x_j\right)\wedge \d f_1 \wedge \d f_2= 0. \]
By Proposition \ref{proposition: Poisson differential}, \eqref{eq: ker deg 4} and \eqref{eq: canonical iso} this is equivalent to
\[ \d \overline{\alpha}\wedge\d f_1\wedge \d f_2 +\left(\sum_{j=1}^4 a_{j}x_j\right)\mu \in \R[[f_1,f_2]]\mu.\]
Then, Lemma \ref{ex: general 1} implies $a_{j}=0$ for $j=1,2,3,4$ and by Lemma \ref{lemma: degree 4 proof} we have 
\[ \star \circ \sharp\left(\iota_{\star^{-1}\d \overline{\alpha}}(\d f_1\wedge \d f_2)\right)=\d \overline{\alpha}\wedge\d f_1\wedge \d f_2=0 . \]
Hence $\wtd{\alpha}=0$ implies $q_1=q_2=0$ which completes the proof.

\subsection{Degree 2}\label{section: homology in degree 2}

We proceed similar as in the previous cases. Consider
\begin{align*}
    \beta=&\ p \zeta_1\wedge \zeta_2 +q\d f_1\wedge \d f_2 +\sum_{i=1}^2  p_i \d (f_1\zeta_i) +q_{i}\d\zeta_i + \d \left(\sum_{j=1}^4 a_{j}x_j\right)\wedge \dif f_1,
\end{align*}
where $p,q,p_i,q_i\in \R[[f_1,f_2]]$ and $a_j\in \R[[x_2^2,x_4]]$. We want to show that
\begin{align*}
    0=\wtd{\beta}:= \beta - \delta_{\pi}(\gamma) \quad \Rightarrow \quad p=q=p_i=q_i=a_j=0 
\end{align*}
Replicating the argument for Lemma \ref{lemma: ses degree 2}, replacing Corollary \ref{lemma: ps are zero with d} with Remark \ref{remark: adapted lemma} we obtain 
\[ p=p_1=p_2=q_1=q_2=0\quad \text{ and } \quad \iota_{\star^{-1}\gamma}(\d f_1), \star^{-1}\d \gamma \in \R[[f_1,f_2]].\] 
As such we are left with
\[ 0=\wtd{\beta}= q\d f_1\wedge \d f_2- \iota_{\pi} (\d \gamma) + \d\left(\iota_{\star^{-1}\gamma}(\d f_2)+\sum_{j=1}^4 a_{j}x_j\right)\wedge \d f_1 -\d\iota_{\star^{-1}\gamma}(\d f_1) \wedge\d f_2.\]
Hence Lemma \ref{lemma: vector field normal form} implies $q=a_j=0$. Counting the elements and comparing them to Proposition \ref{proposition: Poincare series} implies that we have a representative for every cohomology class.

\subsection{Degree 3}
By Proposition \ref{proposition: Poisson differential}, elements in the image of the differential have the form
\begin{align*}
    \dif f_1\wedge \dif f_2 \wedge \dif g,
\end{align*}
for $g\in \SR$. Hence we can argue as in the proof of Lemma \ref{lemma: ses degree 1}.

\section{Proof of Corollary \ref{corollary: deformations}}\label{sec:deformations}

A Poisson structure coming from a deformation of the volume form is equals $g\pi$ for some $g \in \SR$ with positive constant term. Using \eqref{eq: isomorphism form vf} this corresponds to the $2$-form $g \d f_1 \wedge \d f_2$. 

From the definition of the differential we see that 
\[ g\d f_1\wedge \d f_2\ \in \ker \delta_\pi \]
By Proposition \ref{proposition: Poisson differential} and Theorem \ref{main theorem} we can write
\[ g\d f_1\wedge \d f_2 = \beta +\iota_{\pi} (\d \gamma) - \d \iota_{\star^{-1}\gamma}(\d f_1 \wedge\d f_2)\]
for some $\beta=\beta(a_i,p,p_j,q,q_j)$ representing a class in the second Poisson homology according to Theorem \ref{main theorem} and $\gamma\in \Omega^3_f$. Following the first part of the argument from subsection \ref{section: homology in degree 2} we can conclude that $p=p_i=q_i=0$. Using the first part of the proof of Lemma \ref{lemma: vector field normal form} we can conclude that $a_i=0$ and 
\begin{align*}
     \gamma =r_1\epsilon_1+ r_2\epsilon_2+\dif f_1\wedge (r \zeta_1\wedge \zeta_2+ \wtd{\beta} +  \d f_2\wedge \alpha )
\end{align*}
for unique $r_1\in \R[[f_1,f_2]]$, $r_2\in \R[[f_2]]$, some $r\in \R[[f_1,f_2]]$, $\alpha \in \Omega^1_f$ and where $\wtd{\beta}$ represents a non-trivial class in $\mathcal{D}^2(\d f_1,\d f_2)$. Note that this implies in particular, that 
\[ g\d f_1 \wedge \d f_2 = q\d f_1 \wedge \d f_2 +\delta_{\pi} (\gamma)\]
for some $q\in \R[[f_1,f_2]]$ and some $\gamma\in \Omega^3_f$ such that 
\[ \iota_{\star^{-1}\gamma}(\d f_i)= 0.\]

In other words, we can find a vector field $X\in \mathfrak{X}_{f,2}^1$ (see \eqref{eq:filtration}) and $q\in \R[[f_1,f_2]]$ satisfying
\[
    g\pi = q\pi + \d _\pi (X).
\]
This allows us to prove the following which implies Corollary \ref{corollary: deformations}.
\begin{claim}
    Given $g\in \SR$ with positive constant term, there exists a formal diffeomorphism $\phi$ satisfying
    \[ \phi^*(g\pi)= p\pi\]
    for some $p\in \R[[f_1,f_2]]$ with positive constant term.
\end{claim}

To see this, first observe that the flow $\phi_t$ of $X$ satisfies:
\begin{equation}\label{eq:flowstuff} 
\phi_1^*(g \pi) - g\pi = \int_{0}^1 \frac{\d}{\d t} \phi_t^* (g \pi) \d t = [X,g \pi] = - g\d_\pi X + (\Lie_Xg) \pi.
\end{equation}
Let us write $g$ as $ g= \sum_{i=0}^\infty g_i$, where $g_i$ is homogeneous of degree $i$ and $g_0>0$. By induction on $i$ we show that if $g_j\in \R[[f_1,f_2]]$ for all $j<i$ then we can change $g\pi$ such that also $g_i \in \R[[f_1,f_2]]$.  

We start with $i=1$. From cohomology statement we can find $\wtd{X}_1$ homogeneous of degree $2$ such that
\begin{equation}\label{eq: primitiv}
    g_1\pi = \d_\pi \wtd{X}_1
\end{equation} 
Taking $X_1:= \frac{1}{g} \wtd{X}_1 $ we obtain from \eqref{eq:flowstuff} that:
\[ \phi_1^*(g \pi) - g\pi = -g\d_\pi X_1 + \mathcal{O}(|x|^4) =-\d_\pi \wtd{X}_1 +\mathcal{O}(|x|^4) \]
Hence we may assume $g_1=0$. In even degrees \eqref{eq: primitiv} will be of the form
\begin{equation*}
    g_{2j}\pi = q_{2j}\pi + \d_\pi \wtd{X}_{2j}
\end{equation*} 
for some $q_{2j}\in \R[[f_1,f_2]]$ contributing to $q$. Repeating the argument for increasing $i\in \N$ we map $g\pi$ to $q\pi$.

\section{Proof of Corollary \ref{corollary: de-rham cohom}}\label{sec:de-rham cohom}

In degree zero, the result follows from Theorem \ref{main theorem} as $\d$ maps representatives to representatives. 

In degree four, we note that for the image of $\d$ we have by \eqref{eq: rel T, zeta} that
\begin{align*}
     \d \left(\sum_{i=1}^2 p_i \zeta_2 \wedge \d \zeta_i +q_i \d f_1 \wedge \d \zeta_i\right) =&\, 2p_2\mu+ (\partial _x p_1+\partial_yp_2) \d f_1 \wedge \zeta_2 \wedge \d \zeta_1 +(\partial _x  p_2-\partial_yp_1) \d f_1 \wedge \zeta_2 \wedge \d \zeta_2  \\
     =&\, (2p_2 +f_2(\partial _x p_1+\partial_yp_2)+f_1(\partial _x  p_2-\partial_yp_1))\mu
    \end{align*}
which implies that $\d $ is surjective in homology. 

In degree three, note that from the discussion in degree four, the kernel of $\d $ is parameterized by
\[ q_1,q_2, p_1 \in \R[[f_1,f_2]]\]
since the differential operator $2+x\partial_x+y\partial_y$ is invertible on $\R[[x,y]]$. To determine the image of $\d$ let $\beta=\beta(a_i,p,p_j,q,q_j)$ be a representative of a class in the second Poisson homology. We can write $\beta$ as
\[ \beta:= p \zeta_1\wedge \zeta_2 +q\d f_1\wedge \d f_2 + \sum_{i=1}^2 p_i \d f_1\wedge \zeta_i + q_i\d \zeta_i +\sum_{i=1}^4 \d (a_ix_i)\wedge \d f_1\]
for $p, p_i, q, q_i \in  \R[[f_1,f_2]]$ and $a_i \in \R[[x_2^2,x_4]]$. Applying $\d$ and using \eqref{eq: rel E, zeta} and \eqref{eq: rel T, zeta} yields
\begin{align*}
    \d \beta= \left(\partial _x p \d f_1 +\partial _y p \d f_2\right)\wedge \zeta_1\wedge \zeta_2 + 2p \zeta_2\wedge\d \zeta_1+  \sum_{i=1}^2 \partial_y p_i \d f_2 \wedge \d f_1 \wedge \zeta_i\\
    +(\partial_xq_1+\partial_y q_2-p_1) \d f_1\wedge \d \zeta_1 +(\partial_xq_2-\partial_yq_1-p_2) \d f_1\wedge \d \zeta_2\quad 
\end{align*} 
Let $\wtd{q}_1,\wtd{q}_2,\wtd{p}_1\in \R[[f_1,f_2]]$ and 
\[ \wtd{p}_2:= -(2+x\partial_x+y\partial_y)^{-1}(y\partial_x-x\partial_y)\wtd{p}_1.\]
Let $\gamma=\gamma(\wtd{p}_i,\wtd{q}_i)$ be the corresponding representative of a class in $\ker \d$. We want to determine for which $p,p_i, q,q_i, a_i$ as above and $g\in \SR$ we have
\[ 0= \gamma + \d \beta +\d g \wedge \d f_1\wedge \d f_2.\]
Wedging with $\d f_1$ and $\d f_2$ respectively, yields
\begin{align*}
    (f_1^2+f_2^2)\partial_y p +2f_2p +f_1\wtd{p}_2+f_2\wtd{p}_1= 0 = (f_1^2+f_2^2)\partial_x p +2f_2p -f_2\wtd{p}_2+f_1\wtd{p}_1
\end{align*}
and hence we obtain
\[-\wtd{p}_1= 2p+ f_1\partial_x p+f_2\partial_y p. \]
Therefore, we may assume $\wtd{p_1}=\wtd{p_2}=p=0$. Contracting the remaining equation with $T_1$ and $T_2$ respectively, yields using \eqref{eq: rel T, zeta contract} that
\begin{align*}
    0=\partial_xq_2-\partial_yq_1-p_2-f_1\partial_yp_1-f_2\partial_y p_2 +\wtd{q}_1 +4\d g(T_1)\\
    0=\partial_x q_1+\partial_y q_2- p_1-f_2\partial_yp_1+f_1\partial_y p_2 +\wtd{q}_2 +4\d g(T_2)
\end{align*}
By Lemma \ref{lemma: dgt zero} this implies $\d g(T_1)=\d g(T_2)=0$. Hence $\wtd{q}_1, \wtd{q}_2$ can be realized by choices of $p_2$ and $p_1$, respectively.

In degree two, let's denote a representative of the second Poisson homology by $\beta=\beta(\wtd{a}_i,\wtd{q},\wtd{q}_j,\wtd{p},\wtd{p}_j)$ as above. Then by the above discussion, $\beta\in \ker \d$ iff $\wtd{p}=0$ and $\wtd{p}_i$ such that
\[ \wtd{p}_2=(1+y\partial_y)^{-1}\left(\partial_x\wtd{q}_2-\partial_y\wtd{q}_1-f_1\partial_y\wtd{p}_1\right)\quad \text{ and }\quad  \wtd{p}_1= (1+y\partial_y)^{-1}\left(\partial_x \wtd{q}_1+\partial_y \wtd{q}_2+f_1\partial_y \wtd{p}_2\right).\]
We denote by $\alpha=\alpha(a_i,b_i,q_j,p_j)$ a representative of the first Poisson homology and $\gamma\in \Omega^3_f$. We want to find $\alpha$ and $\gamma $ such that
\begin{align*}
    0= \beta+ \d \alpha  +\iota_{\pi} (\d \gamma) - \d \iota_{\star^{-1}\gamma}(\d f_1 \wedge\d f_2)
\end{align*}
where $\d \alpha$ is given by
\begin{align*}
    \d \alpha= \sum_{j=1}^2\d( p_j \zeta_j)+ \left(\partial_xq_2-\partial_y q_1)\right)\d f_1\wedge \d f_2 + \sum_{i=1}^4\d \left( b_ix_i\right)\wedge \d f_1.
\end{align*}
Wedging the equation above with $\d f_1$ and contracting with $T_1$ and $T_2$ yields by \eqref{eq: rel T, zeta} and \eqref{eq: rel T, zeta contract} that
\begin{align*}
    0=&\, f_1\partial_y p_1+f_2\partial_y p_2+p_2 +\wtd{q}_2 +d(\iota_{\star^{-1}\gamma}(\d f_1))(T_1) \\
    0=&\, f_2\partial_y p_1-f_1\partial_y p_2+p_1 +\wtd{q}_1 +d(\iota_{\star^{-1}\gamma}(\d f_1))(T_2)
\end{align*}
Using again Lemma \ref{lemma: dgt zero} implies $d(\iota_{\star^{-1}\gamma}(\d f_1))(T_1)=d(\iota_{\star^{-1}\gamma}(\d f_1))(T_2)=0$. Hence $\wtd{q}_1$ and $\wtd{q}_2$ can be realized by choices of $p_1$ and $p_2$, respectively. That is, we may assume $\wtd{q}_i=\wtd{p}_i=p_i=0$. From Lemma \ref{lemma: vector field normal form} implies that $\wtd{a}_i=b_i$ and $\partial_xq_2=\partial_y q_1$.

Finally, to determine the kernel of $\d$ in degree one, we note that from the discussion in degree two we immediately get $p_1=p_2=b_i=0$ and $\partial_xq_2=\partial_y q_1$. Using Theorem \ref{main theorem} we therefore obtain the statement.

\printbibliography
\Addresses
\end{document}